\documentclass[11pt]{article}
\usepackage[utf8]{inputenc}
\usepackage{amsmath}
\usepackage{amsfonts}
\usepackage{amssymb}
\usepackage{amsthm}
\usepackage{subcaption}
\usepackage{graphicx}
\usepackage{fullpage}
\usepackage{mathtools}
\usepackage{enumerate}
\usepackage{enumitem}
\usepackage{geometry}
\usepackage{pdfpages}
\usepackage{color}
\usepackage{algorithm}
\usepackage{algpseudocode}
\usepackage{booktabs}
\usepackage{commath}
\usepackage[sort]{natbib}
\usepackage{url}
\usepackage{hyperref}
\usepackage{booktabs}
\usepackage{comment}

\usepackage{color}

\newcommand{\xh}{x_{h}}
\newcommand{\yh}{y_{h}}
\newcommand{\xhinv}{x_{h}^{-1}}
\newcommand{\yhinv}{y_{h}^{-1}}

\newcommand{\alphatk}{\alpha_{t}^{k}}
\newcommand{\betatk}{\beta_{t}^{k}}

\newcommand{\OmegaT}{\bar{\Omega}_{T}}

 \providecommand{\keywords}[1]
 {
   \small	
   \textbf{Keywords: } #1
 }

\newtheorem{theorem}{Theorem}[section]

\newtheorem{lemma}[theorem]{Lemma}

\newtheorem{assumption}[theorem]{Assumption}
\newtheorem{definition}[theorem]{Definition}

\newtheorem{proposition}[theorem]{Proposition}
\newtheorem{example}[theorem]{Example}

\algrenewcommand\algorithmicrequire{\textbf{Input:}}
\algrenewcommand\algorithmicensure{\textbf{Output:}}

\newcommand{\rhohat}{\hat{\rho}}
\newcommand{\phihat}{\hat{\phi}}
\newcommand{\uhat}{\hat{u}}

\newcommand{\rhoeps}{\rho^{\epsilon}}
\newcommand{\rhostar}{\rho^{*}}

\newcommand{\rhoteps}{\rho_{t}^{\epsilon}}
\newcommand{\rhotstar}{\rho_{t}^{*}}

\newcommand{\sigmat}{\sigma_{t}}
\newcommand{\sigmatau}{\sigma_{\tau}}

\newcommand{\rhoseps}{\rho_{s}^{\epsilon}}
\newcommand{\rhosstar}{\rho_{s}^{*}}

\newcommand{\rhotaueps}{\rho_{\tau}^{\epsilon}}
\newcommand{\rhotaustar}{\rho_{\tau}^{*}}
\newcommand{\rhotaueeps}{\rho_{\tau-\epsilon}^{\epsilon}}
\newcommand{\rhotauestar}{\rho_{\tau-\epsilon}^{*}}

\newcommand{\pt}{\partial_{t}}

\newcommand{\Rbb}{\mathbb{R}}
\newcommand{\Nbb}{\mathbb{N}}

\newcommand{\uteps}{u_{t}^{\epsilon}}
\newcommand{\utstar}{u_{t}^{*}}
\newcommand{\ustar}{u^{*}}

\newcommand{\ue}{u^{\epsilon}}
\newcommand{\utaustar}{u_{\tau}^{*}}
\newcommand{\utilde}{\tilde{u}}
\newcommand{\vtilde}{\tilde{v}}

\newcommand{\phik}{\phi^k}

\newcommand{\uk}{u^k}

\newcommand{\ukp}{u^{k+1}}

\newcommand{\rhotk}{\rho_{t}^{k}}

\newcommand{\rhotkp}{\rho_{t}^{k+1}}
\newcommand{\phitk}{\phi_{t}^{k}}

\newcommand{\utk}{u_{t}^{k}}

\newcommand{\utkp}{u_{t}^{k+1}}

\definecolor{brickred}{rgb}{0.8, 0.25, 0.33}

\title{Maximum Principle of Optimal Probability Density Control}
\author{Nathan Gaby \thanks{Department of Mathematics and Computer Science, Berry College, Mount Berry, GA 30149, USA. (e-mail: ngaby1@berry.edu).} \and Xiaojing Ye \thanks{Department of Mathematics and Statistics, Georgia State University, Atlanta, GA 30303, USA (e-mail: xye@gsu.edu).}
\footnote{This work is supported in part by National Science Foundation under grants DMS-2307466, DMS-2409868, and DMS-2510830.}
}

\date{}

\begin{document}
\maketitle

\begin{abstract}
     We develop a general theoretical framework for optimal probability density control on standard measure spaces, aimed at addressing large-scale multi-agent control problems. In particular, we establish a maximum principle (MP) for control problems posed on infinite-dimensional spaces of probability distributions and control vector fields. We further derive the Hamilton--Jacobi--Bellman equation for the associated value functional defined on the space of probability distributions. Both results are presented in a concise form and supported by rigorous mathematical analysis, enabling efficient numerical treatment of these problems. Building on the proposed MP, we introduce a scalable numerical algorithm that leverages deep neural networks to handle high-dimensional settings. The effectiveness of the approach is demonstrated through several multi-agent control examples involving domain obstacles and inter-agent interactions.
\end{abstract}

\keywords{
Control Hamiltonian dynamics,
Hamilton--Jacobi--Bellman equation of value functional, 
Maximum principle,
Optimal probability control.
}

\section{Introduction}
\label{sec:intro}

Optimal probability density control has recently gained significant attention due to its increasing prevalence in multi-agent control systems managing swarms of agents, such as drones, robots, and autonomous vehicles \cite{brambilla2013swarm,dorigo2021swarm,feng2020overview}. These problems are increasingly relevant in real-world applications, including large-scale surveillance, time-critical search-and-rescue operations, and infrastructure inspection.

Given the large scale of these systems, a probabilistic description of the state is often the most effective approach \cite{jimenez2020optimal}. Consequently, a standard method for solving these multi-agent control problems is through mean-field modeling, which serves as a continuum approximation of the discrete multi-agent system. Here, the probability density of the agent states represents the entire system. Notably, the state of an agent may reside in a high-dimensional space, as it encompasses combined variables such as location, orientation, velocity, and acceleration. In this regard, optimal probability control and multi-agent control are unified into the same problem framework, and the goal is to find the control vector field that maximizes a given total reward function describing the probability density, or equivalently the collective behaviors of the agents, with consideration of inter-agent collaboration and competitions.

Let $\Omega$ be an open subset of $\Rbb^{d}$, where $d$ is the dimension of state vector (such as the concatenation of location, orientation, acceleration, and etc.) of an agent.
Suppose there are $N$ agents with states $\{x_{i} \in \Omega: i = 1,\dots, N\}$ which follow some initial distribution $p$ on $\Omega$, i.e., $x_{i} \sim p$. 
Under a time-evolving control vector field $u: \Omega \times [0,T] \to \Rbb^{d}$ over a fixed time horizon $[0,T]$, the states of the agents change according to 
\begin{equation}
    \label{eq:agent-ode}
    \dot{x}_{i}(t) = u(x_{i}(t), t), \quad i=1,\dots, N .
\end{equation}
Each agent accumulates some running rewards during $[0,T]$ and receives a terminal reward at time $T$.
Unlike classical optimal control, these rewards \emph{depend not only this specific agent but also the collective behaviors of all agents.} 
We provide a few examples of running and terminal reward functions below.

We use $\rho(\cdot,t)$ to denote the time marginal of the states $x_{i}(t)$ for every $t$, namely, $x_{i}(t) \sim \rho(\cdot, t)$. Then \eqref{eq:agent-ode} corresponds to the continuity equation
\begin{equation}
    \label{eq:ct-eq}
    \partial_{t} \rho(x,t) + \nabla \cdot \bigr(\rho(x, t) u(x,t) \bigr) = 0
\end{equation}
in $\Omega \times [0,T]$. 
For notation simplicity, we write $\rho(\cdot,t)$ and $u(\cdot,t)$ as $\rho_{t}$ and $u_{t}$ respectively, and often omit $(x)$ in the integrals hereafter unless it is needed for clarification.

Consider for example that the agents are required to keep distances from each other to avoid collision. In this case, we can set the running reward received by the $i$th agent at time $t$ as
\begin{equation}
    \label{eq:ri}
    - \frac{1}{2}|u(x_{i}(t),t)|^{2} - \sum_{j \ne i} \frac{\gamma}{|x_{j}(t) - x_{i}(t)|^{2} + \epsilon} ,
\end{equation}
where $\gamma, \epsilon > 0$ are user-selected hyperparameters. In \eqref{eq:ri}, the first term  penalizes the cost of control energy by the $i$th agent, whereas the second term penalizes the closeness to other agents, with $\gamma$ as the penalty weight and $\epsilon$ a tiny threshold value ensuring the denominator is positive. 
Taking sum of \eqref{eq:ri} over $i=1,\dots, N$, dividing by $N$, and letting $N \to \infty$, we find the continuum limit to be a functional of $\rho_{t}$ and $u_{t}$:
\begin{equation}
    \label{eq:r}
    - \frac{1}{2} \int_{\Omega} |u_{t}|^{2}\rho_{t} \, dx - \int_{\Omega}\int_{\Omega} \frac{\gamma}{|y-x|^{2}+\epsilon}\rho_{t}(y)\rho_{t}(x)\, dy dx .
\end{equation}
This is an example of running reward functional $R(\rho_{t}, u_{t})$.

For the terminal reward, a simple example is that we want all agents to gather at a specified state $x^{*} \in \Omega$ at the terminal time $T$. In this case, we can set the terminal reward received by the $i$th agent at $T$ as
\begin{equation}
    \label{eq:gi}
    - \frac{1}{2} |x_{i}(T) - x^{*}|^{2} ,
\end{equation}
which encourages the $i$th agent to get close to the target $x^{*}$.
Again, taking sum of $i=1,\dots,N$, dividing by $N$, and letting $N\to\infty$, we obtain the total terminal reward functional of $\rho_{T}$:
\begin{equation}
    \label{eq:g}
    - \frac{1}{2} \int_{\Omega} |x - x^{*}|^{2} \rho_{T}(x) \, dx ,
\end{equation}
which is an example of terminal reward functional $G(\rho_{T})$.

The running and terminal rewards \eqref{eq:r} and \eqref{eq:g} are typical examples in multi-agent control, whereas other suitable rewards can be designed according to specific applications. 
%
% The framework we propose in this work require some mild regularity conditions of these reward functionals, which are specified in Section \ref{sec:theory}. 

Now we have the optimal probability control problem of interest in its general form as
\begin{subequations}
\label{eq:control-problem}
\begin{align}
    \max_{u\in U_T} \quad & I[u] := \int_0^T R(\rho_t,u_t) \, dt + G(\rho_T) , \label{eq:odc-obj} \\
    \mathrm{s.t.} \quad & \partial_t\rho_t + \nabla \cdot (\rho_t u_t) = 0,\quad \forall\, (t,x) \in [0,T] \times \Omega, \label{eq:odc-ct-eq} \phantom{\frac{}{}} \\
    & \rho_0=p,\quad \forall \, x \in \Omega, \label{eq:odc-init}
\end{align}
\end{subequations}
where $I$ is the total reward functional, $U_{T}$ is the admissible set of control vector fields, and $p$ is a given initial probability distribution, i.e., $x_{i}(0) \sim p$. 
(We assume $I^{*}:= \max_{u \in U_{T}} I[u]$ to be finite in this work.)
We will give the definitions of the terms in \eqref{eq:control-problem} in detail in Section \ref{sec:theory}.

The main goal of this work is to establish a maximum principle of the optimal probability control problem \eqref{eq:control-problem} in Section \ref{subsec:mp}, along with newly defined adjoint partial differential equation (adjoint PDE) and control Hamiltonian functional $H$.
In addition to the maximum principle, we establish the Hamilton--Jacobi--Bellman equation of the value functional associated with \eqref{eq:control-problem} in Section \ref{subsec:hjb}.
Moreover, we develop a numerical algorithm aligned with the proposed maximum principle and demonstrate its performance on several probability (multi-agent) control examples in Section \ref{sec:algorithm}.

The novelty and significance of the results presented in this work are in two phases: (i) These results are rigorously established on the infinite-dimensional space of probability density functions in contrast to finite-dimensional Euclidean space in classic control theory; and (ii) Our theory is built on standard measure space, yielding concise presentation, simple interpretation, and developments of fast numerical algorithms for applications. 

% We can also consider matching a given target probability $p_{1}$ on $\Omega$ at time $T$.
% %
% In this case, we can set the terminal reward functional as
% %
% \begin{equation}
%     \label{eq:g-kl}
%     G(\rho_{T}) := \kl (p_{1}, \rho_{T}) = -H(p_{1}) - \mathbb{E}_{X \sim p_{1}}[\log \rho_{T}(X)]
% \end{equation}
% where $\kl$ is the Kullback--Leibler (KL) divergence between two probability distributions on $\Omega$ and $H(p_{1)}$ is the entropy of $p_{1}$.
% %
% Since $H(p_{1})$ does not depend on $\rho_{T}$, we can readily see \eqref{eq:g-kl} is identical to the continuum
% \begin{equation}
%     \label{eq:evidence}
%     - \sum_{i=1}^{N} \log\int_{\Omega} p_{1}(x) \log \rho_{T}(x) \, dx
% \end{equation}

% In this work, we establish a general mathematical framework of the optimal probability (multi-agent) control problem and a maximum principle.
% %
% The framework includes a general probability control formulation that has a broad range of real-world applications.
% %
% Meanwhile, the maximum principle reduces the optimization over control vector field to a pointwise condition as an analogue to the classic Pontryagin's principle but on the space of probability density functions.
% %
% These results simplify relevant analysis and implementation for many optimal control problems of interest. 

\section{Related Work}
\label{sec:related-work}

Characterizing behaviors of large multi-agent systems using probability distributions is typically called a mean-field model.
The approach of mean-field model to approximating multi-agent system can be traced back to \cite{mckean1967propagation}. Theoretical justifications are provided in \cite{fornasier2014mean-field,fornasier2019mean-field}. 
Mean-field models are also used to establish Nash equilibrium in the case where each agent try to maximal its individual reward depending on the configurations of all the other agents in large multi-agent systems. This is an important research topic, known as the mean field game (MFG) \cite{caines2006large,nourian2013e-nash,lasry2007mean}, which has been intensively studied in the past two decades.

In this section, we present an overview of optimality conditions and numerical algorithms for general probability control problem \eqref{eq:control-problem}. Due to space limitation, we omit results for special cases of \eqref{eq:control-problem}, such as those with control-affine drifts and quadratic running cost. 
% , and optimal probability control problems with fixed endpoint constraint. 
%
% We will, however, discuss some interesting connections between this work and a few other related research topics in Section \ref{subsec:extension}.

A substantial body of theoretical studies on optimality conditions of probability control has been developed in the Wasserstein space of probability measures, grounded in optimal transport theory \cite{ambrosio2003optimal,villani2003topics,villani2008optimal,ambrosio2013users}.
In \cite{bonnet2019pontryagin,bonnet2019pontryagin_esaim,bonnet2021necessary}, a series of studies are conducted on the Pontryagin maximum principle (PMP) of probability control problems on the Wasserstein space, where the optimality condition is formulated using the Wasserstein subdifferential \cite{ambrosio2003optimal,ambrosio2008hamiltonian} and the Hamiltonian, forward-backward continuity equation, and the PMP are defined on the joint probability with one of the two marginals being the probability to be controlled. 
In \cite{bongini2017mean-field}, the optimality conditions for the probability control problem on Wasserstein space are established from the mean-field approach with a coupled ODE-PDE constraint.
In \cite{jimenez2020optimal}, the value function is proved to be the unique viscosity solution of a first-order Hamiliton--Jacobi--Bellman equation under mild regularity conditions in the Wasserstein space.
Such result is also extended to Mayer-type optimal control problems on the Wasserstein space defined over compact Riemannian manifold in \cite{jean2022mayer}.
The conditions for approximate and exact controllability where control is restricted to a local bounded region are investigated in \cite{duprez2019approximate}, which also rely on the theory of Wasserstein spaces such as Wasserstein subdifferential.

By contrast, the present work follows an alternative approach to \eqref{eq:control-problem} which does not involve Wasserstein spaces and metrics. 
The derivations are established on basic measurable spaces and metrics in analysis, such as $L^{2}$ space and metric, yielding concise presentation and simple computations.
This is the main difference between this work and the aforementioned ones in the literature.

This work is partly motivated by the limited existing results on fast numerical methods for solving the optimal probability control \eqref{eq:control-problem}, particularly for high-dimensional problems.
Existing works in this field are mostly concerned about approximating a given target distribution from an initial one without any cost or reward of the time-dependent control vector field. For instance, in \cite{eren2017velocity}, the control field is determined by the difference between the current and target densities, which must be known with explicit forms, and is proved to achieve the density matching goal. 
In \cite{roy2018fokker--planck}, a numerical method for solving the optimal probability control problem with particle randomness, for which the continuity equation is replaced with a Fokker--Planck equation, is considered. % In this work, the system is discretized using an alternate-direction implicit Chang--Cooper scheme and a projected non-linear conjugate gradient scheme is used to solve the optimality system.
In \cite{sinigaglia2022density}, the control function is approximated with one-dimensional radial basis functions, the PDEs of the state and adjoint are discretized in space with finite element method, and the integral in time is approximated by the trapezoidal method. 
Using Wasserstein-2 distance between the current probability, which is represented by the average of Dirac delta functions at the current particle locations, and the target probability is used as the running cost, numerical experiments conducted on one- and two-dimensional cases are presented in \cite{inoue2021optimal}.
In \cite{pogodaev2017numerical}, a numerical method is developed to solve a special probability control problem with continuity equation, where the goal is to maximize the density mass in a target region at terminal time. The method is not applicable to general probability control problems and the numerical experiments are mostly conducted in two-dimensional cases. 
In \cite{breitenbach2020pontryagin}, the PMP is applied to optimal control problems governed by the Fokker--Planck equation, and a finite difference scheme based on the sequential quadratic Hamiltonian (SQH) method is developed for low-dimensional problems. 
The work \cite{hoshino2020finite-horizon} considers the Mayer-type optimal probability control problem where the Wasserstein distance between the final probability and the target probability as a terminal cost. It is a discrete-time version where at each time the state is pushed forward by a control to be solved, and the control is obtained by employing normalizing flows \cite{dinh2014nice:} to minimize the terminal cost. Several numerical examples are provided for two-dimensional problems. 
In \cite{hoshino2023finite-horizon}, an extended version to the case where the dynamics is described by a stochastic differential equation is considered, where maximum principle is established on the state and adjoint variables, not on the probability density and adjoint function, and no numerical result is provided. 
The optimal probability control problem with Wasserstein-1 distance between the target density and controlled density at terminal time as the terminal cost is considered in \cite{ma2023high-dimensional}, where Wasserstein-1 distance is reformulated in its dual form and the control problem reduces to a minimax problem. The control vector field is parameterized as a deep neural network and experiments are performed on a few high-dimensional problems.  
%However, there is no theoretical connection between the numerical algorithm and the control theory in \cite{ma2023high-dimensional}.
%
However, these methods are not applicable to high-dimensional problems due to spatial discretizations or lack control-theoretic guarantees for solution optimality.
%
% Except for \cite{ma2023high-dimensional}, all existing methods do not apply to high-dimensional cases where $d \ge 3$.

% \subsubsection{Main contributions of this work}

% This work consists of two major contributions: (i) We develop a comprehensive theoretical framework for optimal probability control and rigorously prove our results for Lebesgue spaces of measurable functions. Specifically, we establish two fundamental theorems for the probability control problem: the Pontryagin Maximum Principle (PMP) and the Hamilton-Jacobi-Bellman (HJB) equation. (ii) We propose a numerical algorithm aligned with the PMP theory in (i) to compute the optimal control vector field $u$ using reduced-order models, such as deep neural networks (DNNs). We then prove the convergence properties of this algorithm and demonstrate its empirical performance on several high-dimensional probability control problems.

\section{Theory of Optimal Probability Control}
\label{sec:theory}

In this section, we develop the maximum principle and the Hamilton--Jacobi--Bellman (HJB) equation for the optimal probability control problem \eqref{eq:control-problem}. They are analogues to the Pontryagin maximum principle and HJB equation in classical optimal control.
The difference is that the ones introduced in this work are defined on the space of probability density functions, instead of finite-dimensional Euclidean spaces.

In the remainder of this paper, we use $|\cdot|$ to denote the absolute value of scalars and 2-norm of vectors, and $a \cdot b$ as the inner product of two vectors $a$ and $b$.
%
% $\langle f, g\rangle = \int_{\Omega} f(x)g(x)\, dx$, and $\|f\|_{2}=(\int_{\Omega} |f(x)|^{2}\,dx)^{1/2}$.
%
% the function norm where the space is specified by $*$. 
For the $L^2(\Omega)$ space, we use $\|\cdot\|_2$ and $\langle \cdot, \cdot \rangle $ to denote its associated norm and inner product, respectively. 
%
% The interior of a subset $U$ in a metric space is denoted by $\intr(U)$.
%
We use $\nabla$ to denote the gradient (Jacobian) of a scalar-valued (vector-valued) function with respect to $x$, $\nabla\cdot$ is the divergence, and $\partial_*$ and $\nabla_{*}$ for partial and full gradient with respect to the variable at the place of $*$, respectively.
Furthermore, $\frac{\delta}{\delta \rho}G(\rho)$ is the Fr\'{e}chet derivative of a functional $G$ with respect to its variable $\rho$ in the $L^2$ sense.
%
% We denote $[N]:=\set{1,\dots,n}$.
%
% For notation simplicity, we write $\rho(t,\cdot)$ and $u(t,\cdot)$ as $\rho_t$ and $u_t$ respectively, and drop the argument $(x)$ as we did in \eqref{eq:control-problem} unless needed for clarification hereafter.

% In this work, we establish a general maximum principle of optimal probability control. The control problem is defined as follows: Let $\Omega$ be an open subset of $\Rbb^d$, where $d$ denotes the state dimension. 
%(We validate the proposed algorithm on high-dimensional instances, specifically where $d \ge 10$).
%
We define the space of probability density functions on $\Omega$:
\begin{equation} 
\label{eq:P}
P = \Bigl\{ p: \Omega \to [0,\infty) : \int_{\Omega} p(x) \, dx = 1 \Bigr\}.
\end{equation}
%
%For notation simplicity, we omit specifying $\Omega$, $(x)$, and $dx$ hereafter unless there is a chance of confusion.
%
Furthermore, for the given terminal time $T>0$, we define the space of time-evolving probability density functions as:
\begin{equation*} 
%\label{eq:PT}
P_{T} = \bigl\{ \rho : \Omega \times [0,T] \to \Rbb : 
     \rho(\cdot,t) \in P, \ \forall \, t \in [0,T] \bigr\},
\end{equation*}
% where $L^1(0,T)$ is the space of Lebesgue integrable functions defined on $[0,T]$.
%
% (Note: The $L^1$ requirement in time $t$ is a sufficient condition that may be relaxed, as discussed in Section [X]).
%
We assume that the admissible set of control vector field at any fixed time is 
\begin{equation}
\label{eq:U}
U = \bigl\{ w \in C(\bar{\Omega}; \Rbb^d) : \mathrm{Lip}(w) \le M_{U},\, (w \cdot n)|_{\partial \Omega} = 0 \bigr\}, 
\end{equation}
for some $M_{U} > 0$, namely, $w$ is $M_{U}$-Lipschitz continuous in $\bar{\Omega}$ and satisfies the boundary condition $w(x) \cdot n(x) = 0$, where $n(x) \in \Rbb^{d}$ is the outer normal of $\partial \Omega$ at $x$ for all $x \in \partial \Omega$.
%
% where $W_{0}^{1,\infty}(\Omega;\Rbb^{d})$ is the $(1,\infty)$-Sobolev space of vector-valued functions with zero boundary condition on $\partial \Omega$.
%
% Note that \eqref{eq:U} holds if and only if $w$ is essentially bounded by $M_{U}$ (bounded by $M_{U}$ except for a measure zero subset of $\Omega$) and $M_{U}$-Lipschitz continuous on $\Omega$ with $w=0$ on $\partial \Omega$. 
%
The boundary condition of $w$ conserves the total probability density as 1.
Then we define the admissible set of time-evolving control vector fields as
\begin{equation}
\label{eq:UT}
    U_T= \bigl\{ u: \bar{\Omega} \times [0,T] \to \Rbb^{d} : u(\cdot, t) \in U,\ \forall\, t \in [0,T] \bigr\}.
\end{equation}
Note that \eqref{eq:UT} does not require $u(x,\cdot)$ to be a continuous function of $t$, allowing $u$ to change direction instantaneously and yielding broader range of applications of the results obtained in this work.

We have a few assumptions on the probability distributions and controls as follows, which ensure the derivations later can be carried out rigorously.
%
% In addition to the definitions above, we assume $R: P \times R \to \Rbb$ and $G: P \to \Rbb$ are Fr\'{e}chet differentiable with respect to their arguments.  
%
% We also have a differentiability assumption on $u(x,\cdot)\rho(x,\cdot)$, which is a function of $t$ at every fixed $x \in \Omega$, as follows.
%
%\vspace{6pt}

\begin{assumption}
\label{assump:rho}
    We make the following assumptions:
    \begin{enumerate}
        \item[(i)] For all $x \in \Omega$, $\rho(x,\cdot)$ is differentiable in $t$ on $[0,T]$.
        \item[(ii)] For all $t \in [0,T]$, $\rho(\cdot,t) u(\cdot,t)$ is differentiable in $\Omega$. 
        \item[(iii)] For all $x \in \Omega$, $\rho(x,\cdot) u(x,\cdot)$ has Lebesgue points in $(0,T)$ everywhere.
        \item[(iv)] For all $w \in U$, $R(\cdot, w)$ is Fr\'{e}chet differentiable in $P$ in the $L^{2}$ sense.
        \item[(v)] $G(\cdot)$ is Fr\'{e}chet differentiable in $P$ in the $L^{2}$ sense.
    \end{enumerate}
\end{assumption}
%
%\vspace{6pt}

Assumption \ref{assump:rho} (i) and (ii) can be relaxed to weakly differentiable as long as \eqref{eq:ct-eq} is valid.
Assumption \ref{assump:rho} (iii) ensures $\frac{d}{dt} \int_{0}^{t}\rho(x,s)u(x,s)\,ds = \rho(x,t)u(x,t)$ for all $t \in (0,T)$. This is used in a step in the proof of Lemma \ref{lem:uniform-approx} below (we will specify it in the proof). 
Assumption \ref{assump:rho} (iv) and (v) allow functional derivatives of $R$ and $G$ at any $p \in P$, which are needed in the establishment of control Hamiltonian dynamics later.
Notice that running and terminal reward functionals in real-world applications typically satisfy Assumption \ref{assump:rho} (iv) and (v).
For example, the running reward in \eqref{eq:r} has functional derivative
\begin{equation*}
    \frac{\delta}{\delta \rho_{t}} R(\rho_{t},u_{t})(x) = -\frac{1}{2} |u_{t}(x)|^{2} - \int_{\Omega} \frac{\gamma  \rho_{t}(y)}{|y-x|^{2}+\epsilon}\,dy
\end{equation*}
and the terminal reward in \eqref{eq:g} has functional derivative
\begin{equation*}
    \frac{\delta}{\delta \rho_{T}} G(\rho_{T})(x) = - \frac{1}{2} |x - x^{*}|^{2} .
\end{equation*}
We will show experimental results on \eqref{eq:control-problem} with other running and terminal reward functionals in Section \ref{sec:results}.
%
% Assumption \ref{assump:rho} slightly improves the differentiability of $u(x,\cdot)\rho(x,\cdot)$: Since
% \[
% \int_0^T |u_t(x) \rho_t(x)| \,dt  \le M_{U} {\int_0^T \rho_t(x) \,dt } < \infty
% \]
% where the second inequality is due to the integrability requirement in $P_T$, we know $u(x,\cdot)\rho(x,\cdot)\in L^{1}(0,T)$ and hence is differentiable almost everywhere in $(0,T)$. Assumption \ref{assump:rho} requires that it is differentiable everywhere in $(0,T)$.

% Now we have all needed definitions and assumptions for the optimal probability control problem \eqref{eq:control-problem}. We are ready to develop the maximum principle of this control problem.

\subsection{Maximum Principle for Probability Control}
\label{subsec:mp}

We establish a maximum principle (MP) for the optimal probability control problem \eqref{eq:control-problem} with any initial value $p \in P$ in this subsection. As an analogue to the Pontryagin MP for classic optimal control, our MP is a time pointwise necessary condition of the optimal solution $u$ to \eqref{eq:control-problem} defined on the infinite-dimensional spaces $P_{T}$ and $U_{T}$. 
%
% As mentioned before, this is equivalent to finding the optimal control vector field for a massive amount of indistinguishable particles in the mean-field sense. 
%
To develop our MP, we first introduce the definition of adjoint partial differential equation (PDE) and control Hamiltonian functional.
%
%\vspace{6pt}
%
\begin{definition}
    [Adjoint PDE]
    \label{def:adj}
    Let $u \in U_T$ be a control vector field and $\rho \in P_T$ the corresponding evolutionary probability function solving \eqref{eq:odc-ct-eq} with initial \eqref{eq:odc-init}. Define an evolution PDE by
    \begin{equation}
        \label{eq:adj}
        \partial_{t} \phi_t + u_t \cdot \nabla \phi_t = - \frac{\delta}{\delta \rho_{t}} R(\rho_t,u_t) 
    \end{equation}
    in $\Omega \times [0,T]$.
    We call \eqref{eq:adj} the \emph{adjoint PDE} associated to $(\rho,u)$. We call $\phi$ the \emph{adjoint function} of $(\rho,u)$ if $\phi: \Omega \times [0,T] \to \Rbb$ solves \eqref{eq:adj} with terminal condition $\phi_{T} = \frac{\delta}{\delta \rho_{T}} G(\rho_T)$.
\end{definition}
%
%\vspace{6pt}

Throughout this work, we assume $\frac{\delta}{\delta \rho} R(\cdot,w)$, $\frac{\delta}{\delta \rho} G(\cdot)$, and $u$ in the control problem \eqref{eq:control-problem} are sufficiently smooth for any $w \in U$ such that the following assumption holds. 
%
%\vspace{6pt}

\begin{assumption}
\label{assump:phi}
    An adjoint function $\phi: \Omega \times [0,T] \to \Rbb$ for \eqref{eq:control-problem} with any initial $p \in P$ exists.
\end{assumption}
%
%\vspace{6pt}

We see that Assumption \ref{assump:phi} holds for typical application problems, such as those with the running and terminal reward functionals defined in \eqref{eq:r} and \eqref{eq:g}.

Note that the adjoint function $\phi$ has spatial variable $x$ and time variable $t$, and the adjoint equation is an evolution PDE. By contrast, in classical optimal control, the adjoint function is a function of $t$ only, and the adjoint equation is an ODE.

% With the definition of adjoint equation, we can define the Hamiltonian functional as follows.
%
%\vspace{6pt}

\begin{definition}
    [Hamiltonian functional]
    \label{def:Hamiltonian}
    Let $p \in P$, $w \in U$, and $f: \Omega \to \Rbb$ be a differentiable function, we define the \emph{Hamiltonian functional} by
    \begin{equation}
        \label{eq:Hamiltonian}
        H(p, f, w) = \langle p, \ w \cdot \nabla f \rangle + R(p, w).
    \end{equation}
    where $R$ is the running reward functional in \eqref{eq:control-problem}.
\end{definition}
%
%\vspace{6pt}

% Throughout the remainder of this paper, we write $\frac{\delta}{\delta \rho}H(\rho_t, \phi_t, u_t)$ as the first variation of $H$ with respect to its first argument $\rho_t$ in the $L^2$ sense at time $t$. 
%
Now we can verify that $(\rho,\phi)$ induced by any $u \in U_{T}$ satisfies the symplectic control Hamiltonian dynamics, as summarized in the following proposition.
%
%\vspace{6pt}
%
\begin{proposition}
    [Control Hamiltonian dynamics]
    \label{prop:control-Hamiltonian}
    Let $u \in U_T$ be a control vector field, $\rho \in P_T$ the corresponding evolutionary probability, and $\phi$ the adjoint function of $(\rho,u)$. Then $(\rho, \phi)$ satisfies the control Hamiltonian dynamics:
    \begin{equation}
    \label{eq:hamiltonian-dynamics}
        \begin{cases}
            \partial_t \rho_t = \displaystyle\frac{\delta}{\delta \phi_{t}}H(\rho_t, \phi_t, u_t) = - \nabla \cdot (\rho_t u_t) , \\
            \partial_t \phi_t = -\displaystyle\frac{\delta}{\delta \rho_{t}}H(\rho_t, \phi_t, u_t) = - u_t \cdot \nabla \phi_t - \frac{\delta}{\delta \rho_{t}} R(\rho_t,u_t) 
        \end{cases}
    \end{equation}
    for all $t$.
\end{proposition}
%
%\vspace{6pt}

The dynamics \eqref{eq:hamiltonian-dynamics} are easy to verify using \eqref{eq:ct-eq}, \eqref{eq:adj}, and \eqref{eq:Hamiltonian}, and hence we omit the proof of Proposition \ref{prop:control-Hamiltonian}.

In what follows, we write $u^*$ as an optimal control vector field of \eqref{eq:control-problem} and $\rho^*$ the corresponding probability. 
Then we define a needle-like variant $u^{\epsilon}$ of $u^*$ as follows: For any time $\tau \in (0,T)$ and $\epsilon \in (0, \tau)$, and $w \in U$, we define
\begin{equation}
\label{eq:ue}
\ue_t:= \begin{cases}
    w, & \tau - \epsilon < t \le \tau,\\
    u^*_t, & \text{otherwise}.
\end{cases}    
\end{equation}
%
% Note that $w : \Omega \to \Rbb^{d}$ is a fixed vector field, i.e., $w(x)$ is time-independent on $(\tau-\epsilon,\tau)$ for each $x \in \Omega$. Since $ w \in U$, we have $u^{\epsilon}\in U_T$. 
%
We denote by $\rhoteps \in P_T$ the evolutionary probability determined by \eqref{eq:ct-eq} with control vector field set to $u^{\epsilon} \in U_T$ and initial $\rho_{0}^{\epsilon} = p$.
Next, we introduce the perturbation function that describes the first-order approximation of $\rhoteps$ to $\rho_{t}$.

%
%\vspace{6pt}
%
\begin{definition}[Perturbation function]
\label{def:sigma}
    For any $\tau \in (0,T)$, we define the function $\sigma :[0,T]\times \Omega \to \Rbb$ as follows: 
    If $t \in [0,\tau)$, then
    \begin{equation}
        \label{eq:sigma-pre-tau}
        \sigma_t=0 \,;
    \end{equation}
    If $t \in [\tau,T]$, then $\sigma$ is defined as the solution to the initial value problem:
    \begin{equation}
        \label{eq:sigma-ivp}
        \begin{cases}
        \partial_t \sigma_t = - \nabla \cdot \bigl(\utstar \sigmat \bigr) , & \quad \forall\, t > \tau, \\
        \sigmatau = -\nabla \cdot \bigl( \rhotaustar(w- \utaustar) \bigr) . & 
        \end{cases}
    \end{equation}
\end{definition}
%
%\vspace{6pt}

Note that $\sigmat(x)$ is not necessarily differentiable at $t= \tau$ for a fixed $x$, and $\sigmat$ depends on $\tau$ and $w$ but \emph{not} $\epsilon$. 
%
% In this next lemma, we show that $\sigmat$ is the first-order approximation of $\rhoteps$ with respect to $\rhotstar$ in the $L^2$ sense for all $t \in [0,T]$.
%
We show that the first-order approximation of $\rhoteps$ to $\rho_{t}$ is $\sigmat$ in the following lemma.
\begin{lemma}
    \label{lem:uniform-approx}
    Suppose $\tau \in (0,T)$, $\epsilon \in (0,\tau)$, $w \in U$, and $\uteps$ is defined in \eqref{eq:ue}. Let $\sigmat$ be the perturbation function in Definition \ref{def:sigma}. Then for every $t \in [0,T]$, there is
   \begin{equation}
    \label{eq:o-eps}
    \lim_{\epsilon \to 0} \frac{1}{\epsilon} \| \rhoteps - \rhotstar - \epsilon \sigmat \|_2 = 0.
    \end{equation}
\end{lemma}
%
%\vspace{6pt}

\begin{proof}
We prove \eqref{eq:o-eps} in three cases: $t \in [0,\tau)$, $t = \tau$, and $t \in (\tau, T]$. 

For any $t \in [0,\tau)$, let $\epsilon \in (0, \tau -t )$, then there is $t < \tau - \epsilon$. By \eqref{eq:sigma-pre-tau}, we know $\sigmat= 0$. By \eqref{eq:ue}, we have  $\rhoseps = \rhosstar$ for all $s < \tau-\epsilon$. Therefore $\rhoteps = \rhotstar$.

For $t = \tau$, we only need to show that $\sigmatau$ defined by the initial condition in \eqref{eq:sigma-ivp} is equal to $\lim_{\epsilon \to 0} \frac{\rhotaueps - \rhotaustar}{\epsilon}$ almost everywhere in $\Omega$. To this end, let $\psi \in H_0^1(\Omega)$ be arbitrary, then there is
\begin{align}
    \int_{\Omega} \Big( \lim_{\epsilon \to 0} \frac{\rhotaueps - \rhotaustar}{\epsilon} \Big) \psi \,dx 
    = & \ \lim_{\epsilon \to 0} \frac{1}{\epsilon} \int_{\Omega} (\rhotaueps - \rhotaustar) \psi \,dx \nonumber \\
    = & \ \lim_{\epsilon \to 0} \frac{1}{\epsilon} \int_{\Omega} \Big( \int_{\tau-\epsilon}^{\tau} (\partial_{s} \rhoseps - \partial_{s} \rhosstar) \, ds \Big) \psi \, dx \nonumber \\
    = & \ \lim_{\epsilon \to 0} \frac{1}{\epsilon} \int_{\tau-\epsilon}^{\tau} \int_{\Omega} (w\rhoseps - \ustar_{s}\rhosstar)\cdot \nabla \psi  \,dx ds \label{eq:pf-t-tau} \\
    = & \ \int_{\Omega} \Big( \lim_{\epsilon \to 0} \frac{1}{\epsilon} \int_{\tau-\epsilon}^{\tau} (w - \ustar_s)\rhosstar ds \Big) \cdot \nabla \psi\,dx \nonumber \\
    = & \ \int_{\Omega}  (w - \ustar_\tau)\rhotaustar  \cdot \nabla \psi\,dx \nonumber \\
    = & \ - \int_{\Omega} \nabla \cdot \bigl((w - \ustar_\tau)\rhotaustar \bigr) \psi \,dx \nonumber \\
    = & \ \int_{\Omega} \sigmatau \psi \,dx \nonumber 
\end{align}
where the first equality of \eqref{eq:pf-t-tau} is due to the Lebesgue dominated convergence theorem; the second due to $\rhotaueeps = \rhotauestar$ and integration in time; the third by the continuity equations of $\rhoeps$ and $\rhostar$, integration by parts on $\Omega$, and $w, u_{\tau}^{*} \in U$ (so that $(w\cdot n)|_{\partial \Omega} = 0$ and $(u_{\tau}^{*}\cdot n)|_{\partial \Omega} = 0$); the fourth by the Lebesgue dominated convergence theorem; the fifth by Assumption \ref{assump:rho} (iii); the sixth is due to integration by parts and $w, u_{\tau}^{*} \in U$; and the last equality by the initial condition in \eqref{eq:sigma-ivp}.

For $t> \tau$, define $\delta(t):= \frac{1}{2}\| \rhoteps - \rhotstar - \epsilon \sigmat\|_2^2$, then there is
\begin{align}
\dot{\delta}(t)
&= \int_{\Omega} (\rho^{\epsilon}_t-\rho^*_t-\epsilon \sigmat) \ \partial_t(\rho^{\epsilon}_t-\rho^*_t-\epsilon \sigmat) \,dx \nonumber \\
&=  - \int_{\Omega} (\rho^{\epsilon}_t-\rho^*_t-\epsilon \sigmat)\ \nabla\cdot \bigl(u^*_t(\rho^{\epsilon}_t-\rhotstar-\epsilon \sigmat) \bigr)\, dx \nonumber \\
 &= \int_{\Omega}u^*_t\cdot \nabla \Big( \frac{1}{2}|\rho^{\epsilon}_t-\rho^*_t-\epsilon \sigmat|^2 \Big) \, dx \label{eq:delta-pf}\\
 & \leq  d M_{U} \delta(t), \nonumber
\end{align}
where the second equality is due to the continuity equations of $\rhoteps$, $\rhotstar$ and $\sigmat$ on $(\tau,T]$; the third is due to integration by parts and $(u_{t}^{*} \cdot n)|_{\partial \Omega} = 0$; and the last equality is due to $\mathrm{Lip}(u_{t}) \le M_{U}$ yielding $\| \nabla \cdot \utstar\|_{\infty} \le d M_U$.
By Gr\"{o}nwall's inequality, we have $\delta(t) \le \delta(\tau) e^{d M_{U} (T-\tau)} \le \delta(\tau) e^{d M_{U} T}$ for all $t \in (\tau,T]$. 
As we have shown that \eqref{eq:o-eps} holds true at $t = \tau$, namely, $\delta(\tau) = o(\epsilon^2)$, we obtain $\delta(t) = o(\epsilon^{2})$, which implies \eqref{eq:o-eps} for all $t \in (\tau, T]$.
\end{proof}
%
%\vspace{6pt}

In the next lemma, we derive the first variation of $I$, which is critical in the proof of the maximum principle that leverages the optimality of $u^{*}$.
%
%\vspace{6pt}

\begin{lemma}
\label{lem:cost-derivative}
    Let $u^{\epsilon}$ be defined in \eqref{eq:ue} and $I$ defined in \eqref{eq:control-problem} with control vector field $u^{\epsilon}$, then
    \begin{align*}
    \frac{d}{d \epsilon}I[u^{\epsilon}] \Big|_{\epsilon=0}
    & = R(\rho^*_s, w)-R(\rho^*_s,u^*_s) + \Bigl\langle \frac{\delta}{\delta \rho^{*}_{T}} G(\rhostar_{T}), \sigma_{T} \Bigr\rangle +\int_{0}^{T} \Bigl\langle \frac{\delta}{\delta \rho_{t}^{*}} R(\rho^*_t,u^*_t), \sigmat \Bigr\rangle \, dt
    \end{align*} 
    where $\sigmat$ is given in Definition \ref{def:sigma}.
\end{lemma}
%
%\vspace{6pt}

\begin{proof}
    First of all, we have 
    \begin{align}
    I[u^{\epsilon}]
    & = \int_0^{\tau-\epsilon} R(\rho^{\epsilon}_t,u^*_t) \,dt + \int_{\tau-\epsilon}^{\tau} R(\rho^{\epsilon}_t,w) \, dt + \int_{\tau}^{T} R(\rho^{\epsilon}_t,u^*_t) \,dt + G(\rho^{\epsilon}_T). \label{eq:I-decompose}
    \end{align} 
    Differentiating with respect to $\epsilon$, using Lemma \ref{lem:uniform-approx}, and setting $\epsilon = 0$, we find that
\begin{align*}
        \frac{d}{d \epsilon}I[u^{\epsilon}] \Big|_{\epsilon=0} 
        & = \int_{0}^{\tau} \Bigl\langle \frac{\delta}{\delta \rho^{*}_{t}} R(\rho^*_t,u^*_t), \sigmat \Bigr\rangle \, dt -R(\rhotaustar,\utaustar) + R(\rhotaustar,w) \\
        & \qquad + \int_{\tau}^{T} \Bigl\langle \frac{\delta}{\delta \rho^{*}_{t}} R(\rho^*_t,u^*_t), \sigmat \Bigr\rangle \, dt + \Bigl\langle \frac{\delta}{\delta \rho^{*}_{T}} G(\rhostar_{T}), \sigma_{T} \Bigr\rangle \\
        & = R(\rhotaustar,w) - R(\rhotaustar,\utaustar) + \int_{0}^{T} \Bigl\langle \frac{\delta}{\delta \rho^{*}_{t}} R(\rho^*_t,u^*_t), \sigmat \Bigr\rangle \, dt + \Bigl\langle \frac{\delta}{\delta \rho^{*}_{T}} G(\rhostar_{T}), \sigma_{T} \Bigr\rangle ,
\end{align*}
    which completes the proof. 
\end{proof}
%
%\vspace{6pt}

Now we are ready to establish the maximum principle which provides a time pointwise necessary condition of any optimal control vector field $u^{*}$. 
%
%\vspace{6pt}

\begin{theorem}[Maximum Principle]
    \label{thm:maximum-principle}
    Let $u^*$ be an optimal solution to \eqref{eq:control-problem} with corresponding density $\rho^*$ and $\phi^{*}$ the adjoint function of $(\rho^{*}, u^{*})$. Then $(\rho^{*}, \phi^{*}, u^{*})$ satisfies the control Hamiltonian dynamics \eqref{eq:hamiltonian-dynamics} with $\rho_{0}^{*}=p$ and $\phi^{*}_{T} = \frac{\delta}{\delta \rho_{T}} G(\rho_{T})$, and the following holds for all $\tau \in [0,T]$:
    \begin{equation}
    \label{eq:pmp}
        H(\rhotaustar,\phi_\tau^*,u^*_\tau)\ = \ \max_{w \in U} \, H(\rho^*_{\tau},\phi_{\tau}^*,w) .
    \end{equation}
    % where $\phi^*: [0,T] \times \Omega$ is the adjoint function associated with $(\rho^*,u^*)$ satisfying \eqref{eq:adj}. 
    % In addition, if $\ustar \in \intr(U_T)$, then $H(\rhotstar, \phi_t^*, \utstar)$ is constant in time $t$.
\end{theorem}
%
%\vspace{6pt}

\begin{proof}
    Let $\tau \in [0,T]$ and $w \in U$ and define $u^{\epsilon}$ as in \eqref{eq:ue} with $\tau$ and $w$. Noting that $u^*$ is optimal for \eqref{eq:control-problem} and applying Lemma \ref{lem:cost-derivative}, we find
    \begin{align}
    0 \ge \frac{d}{d \epsilon}I[u^{\epsilon}] \big|_{\epsilon=0} =R(\rho^*_{\tau}, w)-R(\rho^*_{\tau},u^*_{\tau})+  \Bigl\langle \frac{\delta}{\delta \rho^{*}_{T}}  G(\rho^*_T), \sigma_T \Bigr\rangle +\int_0^T\Bigl\langle \frac{\delta}{\delta \rho^{*}_{t}} R(\rho^*_t,u^*_t), \sigma_t \Bigr\rangle \, dt . \label{eq:cost-derivative}
    \end{align} 
    On the other hand, we have
    \begin{align}
    \label{eq:phi-sigma} 
    \langle \phi^*_T,\sigma_T \rangle 
    = \langle \phi^*_{\tau},\sigma_{\tau} \rangle + \int_{\tau}^T \Bigl(\langle \partial_t \phi^*_t,\sigmat\rangle + \langle\phi^*_t,\partial_t \sigmat \rangle \Bigr) \, dt .
    \end{align}
    Substituting the adjoint PDE \eqref{eq:adj} for $(\rhostar,\ustar)$ and the conditions \eqref{eq:sigma-ivp} for $\sigma$ into \eqref{eq:phi-sigma}, and integrating by parts, we obtain
    \begin{align}
        \langle \phi^*_T,\sigma_T \rangle 
    &=\langle \phi^*_{\tau},\sigma_{\tau} \rangle -\int_{\tau}^T \Bigl\langle \frac{\delta}{\delta \rho^{*}_{t}} R(\rho_t^*,u_t^*),\sigmat \Bigr\rangle \, dt \label{eq:phi-sigma-new} \\
    &=\langle \phi^*_{\tau},\sigma_{\tau} \rangle -\int_{0}^T \Bigl\langle \frac{\delta}{\delta \rho^{*}_{t}} R(\rho_t^*,u_t^*),\sigmat \Bigr\rangle \, dt, \nonumber 
    \end{align}
    where the second equality is due to $\sigmat = 0$ for all $t \in [0,\tau)$. 
    
    We combine \eqref{eq:cost-derivative} and \eqref{eq:phi-sigma-new} as well as the terminal condition $\phi_T^* = \frac{\delta}{\delta \rho^{*}_{T}}  G(\rho_T^*)$ to conclude
    \begin{equation}
    \label{eq:opt-first}
        0 \ge  R(\rho^*_{\tau},w)-R(\rho^*_{\tau},u^*_{\tau})+\langle \phi^*_{\tau},\sigma_{\tau} \rangle .
    \end{equation}
    On the other hand, we have
    \begin{align}
    \langle \phi^*_{\tau},\sigma_{\tau} \rangle 
    & = \bigl\langle \phi^*_{\tau},-\nabla\cdot \bigl(w-u_{\tau}^*)\rho_{\tau}^*\bigr) \bigr\rangle \label{eq:phi-sigma-at-tau} \\
    & = \langle \utaustar \cdot \nabla \phi^*_{\tau},\rho_{\tau}^* \rangle - \langle w \cdot \nabla \phi^*_{\tau},\rho_{\tau}^* \rangle, \nonumber
    \end{align}
    where the first equality is due to the value of $\sigmatau$ in the initial condition in \eqref{eq:sigma-ivp}, and the last equality is due to integration by parts.
    Combining \eqref{eq:opt-first} and \eqref{eq:phi-sigma-at-tau}, we have
    \begin{align*}
        0 
        & \ge  \bigl( R(\rho^*_{\tau},w)+ \langle \utaustar \cdot \nabla \phi^*_{\tau},\rho_{\tau}^* \rangle \bigr) - \bigl(R(\rho^*_{\tau},u^*_{\tau}) + \langle w \cdot \nabla \phi^*_{\tau},\rho_{\tau}^* \rangle \bigr) \\
        & = H(\rhotaustar,\phi_{\tau}^*,w) - H(\rhotaustar,\phi_{\tau}^*,\utaustar) .
    \end{align*}
    As this is true for each $ \tau \in [0,T]$ and $w, \utaustar \in U$, the claim \eqref{eq:pmp} is proved.
    % Lastly, to show $H(\rhotstar, \phi_t^*, \utstar)$ is constant in $t$ when $\ustar \in \intr(U_T)$, we have 
    % \begin{align*}
    %     \frac{d}{dt} H(\rhotstar, \phi_t^*, \utstar) 
    %     & = \langle \frac{\delta}{\delta \rho} H(\rhotstar, \phi_t^*, \utstar), \partial_t \rhotstar \rangle 
    %     + \langle \delta_\phi H(\rhotstar, \phi_t^*, \utstar), \partial_t \phi_t^* \rangle 
    %     + \langle \frac{\delta}{\delta u} H(\rhotstar, \phi_t^*, \utstar), \partial_t \utstar \rangle \\
    %     & = \langle -\partial_t \phi_t^*, \partial_t \rhotstar \rangle 
    %     + \langle \partial_t \rhotstar, \partial_t \phi_t^* \rangle + 0  \\
    %     & = 0,
    % \end{align*}
    % where the second equality is due to the property of the control Hamiltonian system \eqref{eq:Hamiltonian} applied to the first two terms, while the last term vanishes is due to the optimality of $\utstar$ justified in \eqref{eq:pmp} that implies $\frac{\delta}{\delta u} H(\rhotstar, \phi_t^*, \utstar) = 0$ when $\ustar \in \intr(U_T)$. This completes the proof.
\end{proof}
%
%\vspace{6pt}

\begin{example}
\label{ex:mp}
    Let $T>0$ be arbitrary and fixed, and $\Omega = \Rbb^{d}$. Consider the control problem \eqref{eq:control-problem} with running reward
    \begin{equation}
        \label{eq:ex-R}
        R(\rho_{t}, u_{t}) = - \frac{1}{2} \int_{\Omega} |u_{t}(x)|^{2} \rho_{t}(x) \, dx
    \end{equation}
    and terminal reward
    \begin{equation}
        \label{eq:ex-G}
        G(\rho_{T}) = - \frac{1}{2} \int_{\Omega} |x|^{2} \rho_{T}(x) \, dx .
    \end{equation}
    We can choose any positive $p \in P$, namely $p(x)>0$ for all $x \in \Omega$, as the initial $\rho_{0}$ in \eqref{eq:control-problem}.

    We can set the terminal reward \eqref{eq:ex-G} to $G(\rho_{T}) = - \frac{1}{2} \int_{\Omega} |x - x^{*}|^{2} \rho_{T}(x) \, dx$ for any $x^{*}$, and the calculations below are similar. We use $x^{*}=0$ for simplicity here.

    By Definition \ref{def:Hamiltonian}, the Hamiltonian functional is
    \begin{equation}
        \label{eq:ex-H}
        H(\rho_{t},\phi_{t}, u_{t}) = \langle \rho_{t}, u_{t} \cdot \nabla \phi_{t}\rangle - \frac{1}{2} \int_{\Omega} |u_{t}|^{2} \rho_{t} \, dx .
    \end{equation}
    From \eqref{eq:hamiltonian-dynamics}, we have the control Hamiltonian dynamics of $(\rho,\phi)$ as
    \begin{equation}
        \label{eq:ex-HD}
        \begin{cases}
            \partial_{t} \rho_{t} = - \nabla \cdot (\rho_{t} u_{t}) , \\
            \partial_{t} \phi_{t} = - u_{t} \cdot \phi_{t} + \frac{1}{2} |u_{t}|^{2} ,
        \end{cases}
    \end{equation}
    with initial value of $\rho$ as
    \begin{equation}
        \label{eq:ex-rho-initial}
        \rho_{0} = p
    \end{equation}
    and terminal value of $\phi$ as
    \begin{equation}
        \label{eq:ex-phi-terminal}
        \phi_{T} = \frac{\delta}{\delta \rho_{T}}G(\rho_{T}) = - \frac{1}{2}|x|^{2} .
    \end{equation}

    By the maximum principle \eqref{eq:pmp}, there is
    \begin{align}
        H(\rho_{t}, \phi_{t}, u_{t}) 
        & = \max_{w \in U} H(\rho_{t}, \phi_{t}, w) \label{eq:ex-control-H} \\
        & = \max_{w \in U} \Bigl\{ \langle \rho_{t}, w \cdot \nabla \phi_{t} \rangle - \frac{1}{2}\int_{\Omega}|w|^{2}\rho_{t}\, dx\Bigr\} . \nonumber
    \end{align}
    Completing the square in the maximization above, we deduce that the maximizer $u_{t}$ satisfies $\rho_{t}(u_{t} - \nabla \phi_{t}) = 0$.
    Assuming $\rho_{t} > 0$, we have
    \begin{equation}
        \label{eq:mp-u}
        u_{t}  = \nabla \phi_{t}.
    \end{equation}
    Plugging \eqref{eq:mp-u} into the control Hamiltonian dynamics \eqref{eq:ex-HD}, and combining with \eqref{eq:ex-rho-initial} and \eqref{eq:ex-phi-terminal}, we know that the optimal $\rho$ solves the initial value problem 
    \begin{equation}
        \label{eq:ex-rho-ivp}
        \begin{cases}
            \partial_{t} \rho_{t} + \nabla \cdot (\rho_{t} \nabla \phi_{t}) = 0, & \forall \, t \in [0,T], \\
            \rho_{0} = p , & 
        \end{cases}
    \end{equation}
    and its corresponding adjoint function $\phi$ solves the terminal value problem
    \begin{equation}
        \label{eq:ex-phi-tvp}
        \begin{cases}
            \partial_{t}\phi_{t} + \frac{1}{2} |\nabla \phi_{t}|^{2} = 0, & \forall\, t \in [0,T], \\
            \phi_{T} = - \frac{1}{2}|x|^{2} .  &
        \end{cases}
    \end{equation}
    We can verify that the solution to \eqref{eq:ex-phi-tvp} is 
    \begin{equation}
        \label{eq:ex-phi-sol}
        \phi_{t}(x) = \frac{|x|^{2}}{2(t-T-1)} .
    \end{equation}
    This implies that the optimal control vector field is
    \begin{equation}
        \label{eq:ex-u-sol}
        u_{t}(x) = \frac{x}{t - T - 1}
    \end{equation}
    for all $x$ and $t$. We will continue with this example to see its value functional and HJB equation in the next subsection.
\end{example}

\subsection{HJB Equation for Probability Control}
\label{subsec:hjb}

In this subsection, we derive the Hamilton--Jacobi--Bellman (HJB) equation associated with the optimal probability control problem \eqref{eq:control-problem}.
To this end, we need to define value functional, which is an analogue to the value function in the classic optimal control.
Note that this is a time-evolving functional defined on the cross space $P \times [0,T]$. 
%
% Therefore, we need to assume that the functional derivative in $P$ and time derivative exist.
%
The definition of this value functional is as follows.
\begin{definition}[Value functional]
    \label{def:value-fn}
    For any $p \in P$ and $t \in [0,T]$, define the time-evolving functional $V$ by
    \begin{equation}
    \label{eq:value-fn}
        V(p,t)=\sup_{u\in U_{T}} \Big\{ \int_t^T R(\rho_s,u_s) \, ds + G(\rho_T) \Bigr\}
    \end{equation}
    where $\rho_{s}$ solves the initial value problem:
    \begin{equation}
        \label{eq:value-fn-ivp}
        \begin{cases}
            \partial_s \rho_s + \nabla \cdot ( u_s \rho_s) = 0, \quad \forall\, s \in [t,T], & \\
            \rho_t = p. &
        \end{cases}
    \end{equation}
    We call $H$ the \emph{value functional} of \eqref{eq:control-problem}.
\end{definition}
%
%\vspace{6pt}

Note that, for any $t$, $V(\cdot,t): P \to \Rbb$ is a functional on the infinite-dimensional space of probability density functions on $\Omega$.
By contrast, value function in classical optimal control is function on some subset of finite-dimensional Euclidean space.

Now we are ready to establish the HJB equation of the value functional $V$ as follows.

%
%\vspace{6pt}

\begin{theorem}[HJB equation for probability control]
    \label{thm:hjb}
    If $V(\cdot,t)$ is continuously Fr\'{e}chet differentiable in $P$ for all $t$ and $V(p,\cdot)$ is $C^{1}$ in $[0,T]$ for all $p \in P$, then $V: [0,T] \times P \to \Rbb$ satisfies
    \begin{equation}
        \label{eq:hjb}
        \partial_t V(p,t)+\max_{w \in U } \Bigl\{ \Bigl\langle w \cdot \nabla  \frac{\delta}{\delta p} V(p,t), p \Bigr\rangle + R(p,w) \Bigr\} =0,    
    \end{equation}
    with terminal value $V(\cdot,T)=G(\cdot)$ in $\Omega$.
\end{theorem}

%
%\vspace{6pt}

\begin{proof}
    It is clear that $V(p,T) = G(p)$ for any $p \in P$ by \eqref{eq:value-fn} and \eqref{eq:value-fn-ivp}.
    Now let $ p \in P$, $t \in [0,T)$. For any $\epsilon > 0$, we consider the constant control $w\in U$ in the interval $[t,t+ \epsilon] \subset [0,T)$ and the associated evolution $\rho_s$ given by 
    \[
    \begin{cases}
        \partial_s \rho_s(x) + \nabla\cdot(w \rho_s)=0, & \forall \, s \in (t,t+\epsilon),\\
        \rho_t = p. &
    \end{cases}
    \]
    From the definition of the value functional \eqref{eq:value-fn}, we know
    \[
    V(p,t) \ge \int_t^{t+\epsilon}R(\rho_s,w) \, ds + V(\rho_{t+\epsilon}, t+\epsilon) .
    \]
    We can rearrange the inequality above and divide by $\epsilon$ to find
    \[
    0 \ge \frac{1}{\epsilon} \int_t^{t+\epsilon}R(\rho_s, w)\,ds + \frac{V(\rho_{t+\epsilon}, t+\epsilon)-V(p,t)}{\epsilon}.
    \] Taking $\epsilon \to 0$ and noticing $\rho_t = p$, we have
    \[
    0 \ge R(p,w)+\partial_t V(p,t)+\Bigl\langle \frac{\delta}{\delta p} V(p,t),\partial_t \rho_t \Bigr\rangle.
    \]
    Realizing $\partial_t \rho_t = -\nabla\cdot (w \rho_t) = - \nabla \cdot (w p)$ and using integration by parts, we have
    \[
    0 \ge R(p,w)+\partial_t V(p,t)+\Bigl\langle w \cdot \nabla \frac{\delta}{\delta p} V(p,t), p \Bigr\rangle.
    \] 
    As we chose $w$ arbitrarily, we can conclude
    \begin{equation}
        \label{eq:value-geq-zero}
        0 \ge \partial_t V(p,t)+\max_{w \in U}\Bigl\{\Bigl\langle w \cdot \nabla \frac{\delta}{\delta p} V(p,t),p \Bigr\rangle + R(p, w)\Bigr\}.
    \end{equation}

    Now we show the inequality in \eqref{eq:value-geq-zero} is actually an equality. To this end, we consider any $(t,p) \in [0,T) \times P$, and let $\ustar : [t,T] \times \Omega \to \Rbb$ and $\rhostar: [t,T] \times \Omega \to \Rbb$ be the optimal solution to the control problem defined by the value functional at $V(p,t)$ in \eqref{eq:value-fn}. Then for any $\epsilon>0$ we know by the optimality of $\rhostar$ and $\rhotstar = p$, there is
    \[
    V(\rhotstar,t)=\int_t^{t+\epsilon} R(\rho_s^*,u_s^*)\, ds+V(\rho_{t+\epsilon}^*, t+\epsilon).
    \] 
    We follow the same steps as before by dividing by $\epsilon$, taking a limit $\epsilon \to 0$, and using integration by parts to find
    \begin{equation}
        \label{eq:value-leq-zero}
        0=\partial_t V(\rhotstar,t)+\Bigl\langle \utstar \cdot \nabla \frac{\delta}{\delta p} V(\rhotstar,t), \rhotstar \Bigr\rangle+R(\rhotstar,u_t^*).
    \end{equation}
   Combining both \eqref{eq:value-geq-zero} and \eqref{eq:value-leq-zero}, $\rhotstar = p$, and $\utstar \in U$, we conclude the maximum can be attained and the equality in \eqref{eq:value-geq-zero} holds. This completes the proof.
\end{proof}
%
%\vspace{6pt}

\begin{example}
    \label{ex:hjb}
    We continue with Example \ref{ex:mp} with the running reward \eqref{eq:ex-R}, terminal reward \eqref{eq:ex-G}, and arbitrary initial value $\rho_{0} = p>0$. 

    Recall from \eqref{eq:hjb} that the HJB equation is
    \begin{equation}
        \label{eq:ex-hjb}
        \partial_{t} V(p,t) + \max_{w \in U} \Bigl\{ \Bigl\langle p, w \cdot \nabla \frac{\delta}{\delta p}V(p,t) \Bigr\rangle - \frac{1}{2} \int_{\Omega} |w|^{2}p \, dx \Bigr\} = 0 .
    \end{equation}
    Inspired by the square completion in \eqref{eq:ex-control-H} and the optimal solution \eqref{eq:ex-u-sol}, we attempt $w = \frac{x}{t-T-1}$ as the maximizer in \eqref{eq:ex-hjb}. Completing the square, we obtain $\nabla \frac{\delta}{\delta p} V(p,t) = w$, from which we deduce that
    \begin{equation}
        \label{eq:ex-value-fn}
        V(p,t) = \frac{1}{2(t-T-1)} \int_{\Omega} |x|^{2} p(x) \, dx .
    \end{equation}
    We can verify that $V(p,t)$ in \eqref{eq:ex-value-fn} indeed satisfies the HJB equation \eqref{eq:ex-hjb} with the maximum attained at $w = \frac{x}{t-T-1}$ for all $t$. Moreover, the terminal value of $V(p,t)$ is
    \begin{equation}
        \label{eq:ex-value-fn-terminal}
        V(p,T) = - \frac{1}{2} \int_{\Omega} |x|^{2} p(x) \, dx  = G(p).
    \end{equation}
    Notice that \eqref{eq:ex-value-fn} also suggests the optimal control to be
    \begin{equation}
        \label{eq:ex-u-sol-value}
        u_{t}(x) = \nabla \frac{\delta}{\delta \rho_{t}} V(\rho_{t},t) = \frac{x}{t-T-1},
    \end{equation}
    which agrees with \eqref{eq:ex-u-sol}.
\end{example}

\subsection{Variations of Initial Perturbation}
\label{subsec:var-init}

In this subsection, we consider the Fr\'{e}chet derivative of $G(\rho_{T})$ with respect to the initial value $p$, where $\rho: \Omega \times [0,T] \to \Rbb$ solves the continuity equation \eqref{eq:ct-eq} with initial value $\rho_{0} = p$.
The result demonstrates an important property of the adjoint function $\phi$.

Suppose the initial states $\{x_{i} \in \Omega: i = 1,\dots,N\}$ of agents are perturbed, changing the initial probability $p$ to
\begin{equation}
    \label{eq:perturb-init}
    p^{\epsilon} := p + \epsilon h + o(\epsilon) \in P, 
\end{equation}
with some $h, o(\epsilon): \Omega \to \Rbb$ satisfying
\begin{equation}
    \label{eq:h-condition}
    \int_{\Omega} h(x)\,dx = 0, \quad \int_{\Omega} o(\epsilon) \, dx = 0 .
\end{equation}
Note that \eqref{eq:h-condition} is necessary for $p^{\epsilon} \in P$.
% and $p^{\epsilon} \in P$. 
%
% Note that $h(x)=O(p(x))$ locally near $x$ with $p(x) \to 0$, as $p$ is approaching to the boundary of $P$ in \eqref{eq:P}, which is a bounded convex subset of $L^{2}(\Omega)$, and the neighborhood of $p$ becomes smaller.
%
Let $\rhoeps$ solve the initial value problem:
\begin{equation}
    \label{eq:var-ivp}
    \begin{cases}
        \partial_{t} \rhoteps + \nabla \cdot ( \rhoteps u) = 0, & \forall\, t \in [0,T] , \\
        \rhoeps_{0} = p^{\epsilon}. &
    \end{cases}
\end{equation}
We are interested in the difference between $\rhoteps$ from \eqref{eq:var-ivp} and $\rho_{t}$ from \eqref{eq:odc-ct-eq}--\eqref{eq:odc-init}, which use the same control vector field $u \in U_{T}$ but different initial values $p^{\epsilon}$ and $p$, respectively.

%
%\vspace{6pt}

\begin{lemma}
    \label{lem:var-g}
    Let $\rhoteps$ and $\rho_{t}$ solve \eqref{eq:var-ivp} with initials $p^{\epsilon}$ and $p$, respectively. Then 
    \begin{equation}
        \label{eq:var-rho-limit}
        \lim_{\epsilon \to 0} \sup_{0 \le t \le  T} \frac{1}{\epsilon} \| \rhoteps - \rho_{t} - \epsilon s_{t} \|_{2} = 0,
    \end{equation}
    where $s_{t}$ solves the initial value problem:
    \begin{equation}
        \label{eq:var-s-ivp}
        \begin{cases}
            \partial_{t} s_{t} + \nabla \cdot ( s_{t} u_{t}) =  0, & \forall\, t \in [0,T], \\
            s_{0} = h. &
        \end{cases}
    \end{equation}
\end{lemma}

%\vspace{6pt}

\begin{proof}
    Define $\delta(t) := \frac{1}{2}\| \rhoteps - \rho_{t} - \epsilon s_{t} \|_{2}^{2}$. Mimicking \eqref{eq:delta-pf} with $\sigma_{t}$ replaced by $s_{t}$, we obtain $\delta(t) \le \delta(0) e^{d M_{U}T}$ for all $t \in [0,T]$.
    Since $\delta(0) = \frac{1}{2}\|p^{\epsilon} - p \|_{2}^{2} = o(\epsilon^{2})$, we know 
    \begin{equation*}
        \lim_{\epsilon \to 0} \sup_{0 \le t \le T} \frac{(2\delta(t))^{1/2}}{\epsilon} \le \lim_{\epsilon \to 0} \frac{(2\delta(0))^{1/2}}{\epsilon} e^{d M_{U} T/2}= 0, 
    \end{equation*}
    which justifies \eqref{eq:var-rho-limit}.
\end{proof}

%\vspace{6pt}

\begin{proposition}
    \label{prop:var-g}
    For any $T > 0$, $u \in U_{T}$, initial value $p \in P$, let $\rho_{t}$ be solution to the initial value problem 
    \begin{equation}
        \label{eq:ex-var-rho-ivp}
        \begin{cases}
            \partial_{t} \rho_{t} + \nabla \cdot (\rho_{t} u_{t}) = 0, & \forall\, t \in [0,T] , \\
            \rho_{0} = p .
        \end{cases}
    \end{equation}
    Then for any functional $G: P \to \Rbb$ Fr\'{e}chet differentiable in $P$, there is
    \begin{equation}
        \label{eq:var-g}
        \frac{\delta}{\delta p} G(\rho_{T}) = \phi_{0} ,
    \end{equation}
    where $p=\rho_{0}$ is the initial value of $\rho$, and $\phi$ solves the terminal value problem:
    \begin{equation}
        \label{eq:var-phi-tvp}
        \begin{cases}
            \partial_{t} \phi_{t} + u_{t} \cdot \nabla \phi_{t} = 0, & \forall\, t \in [0,T], \\
            \phi_{T} = \frac{\delta}{\delta \rho_{T}} G(\rho_{T}) .
        \end{cases}
    \end{equation}
\end{proposition}
%
%\vspace{6pt}

\begin{proof}
    For any $h$ with \eqref{eq:h-condition}, let $p^{\epsilon}$ be the initial value defined in \eqref{eq:perturb-init} and $\rho^{\epsilon}$ solve the initial value problem \eqref{eq:var-ivp}, we know by Lemma \ref{lem:var-g} that  
    \begin{equation}
        G(\rho_{T}^{\epsilon}) = G\bigr(\rho_{T} + \epsilon s_{T} + o(\epsilon)\bigr) ,
    \end{equation}
    where $s_{t}$ is the solution to \eqref{eq:var-s-ivp}.
    Therefore
    \begin{align}
        \label{eq:var-G-terminal}
        \frac{d}{d \epsilon} G(\rho_{T}^{\epsilon}) \Big|_{\epsilon = 0} = \Bigl\langle \frac{\delta}{\delta \rho_{T}} G(\rho_{T}), s_{T} \Bigr\rangle = \langle\phi_{T}, s_{T}\rangle ,
    \end{align}
    where the second equality is due to the terminal condition of $\phi_{t}$ in \eqref{eq:var-phi-tvp}. 
    
    On the other hand, notice that
    \begin{align*}
        \frac{d}{dt} \langle\phi_{t}, s_{t}\rangle
        & = \langle \partial_{t}\phi_{t}, s_{t}\rangle + \langle\phi_{t}, \partial_{t} s_{t}\rangle \\
        & = - \langle u_{t} \cdot \nabla \phi_{t}, s_{t}\rangle - \langle \phi_{t}, \nabla \cdot (s_{t}u_{t}) \rangle \\
        & = 0 ,
    \end{align*}
    where the second equality is due to the PDE of $\phi$ in \eqref{eq:var-phi-tvp} and the PDE of $s$ in \eqref{eq:var-s-ivp}; and the third equality is due to integration by parts and $(u_{t}\cdot n)|_{\partial \Omega}=0$ on $\partial \Omega$.
    Therefore, $\langle\phi_{t}, s_{t}\rangle$ is constant in $t$ and hence 
    \begin{equation}
        \label{eq:var-ip-0T}
        \langle\phi_{T}, s_{T}\rangle = \langle\phi_{0}, s_{0}\rangle = \langle\phi_{0}, h\rangle .
    \end{equation}

    Combining \eqref{eq:var-G-terminal} and \eqref{eq:var-ip-0T}, and noticing that $\frac{\delta}{\delta p} G(\rho_{T}) = \frac{d}{d \epsilon} G(\rho_{T}^{\epsilon})|_{\epsilon = 0}$, we obtain 
    \begin{equation}
        \Bigl\langle\frac{\delta}{\delta p} G(\rho_{T}), h \Bigr\rangle = \langle\phi_{0}, h\rangle .
    \end{equation}
    Since $h$ is arbitrary, we know \eqref{eq:var-g} holds.
\end{proof}
%
%\vspace{6pt}

\begin{example}
    \label{ex:fn-var}
    Let $T>0$ and $p \in P$ be given arbitrarily and $\Omega = \Rbb^{d}$. Define $u \in U_{T}$ by
    \begin{equation}
        \label{eq:ex-var-u}
        u_{t}(x) = \frac{x}{2(t-T-1)} 
    \end{equation}
    for all $x \in {\Omega}$ and $t \in [0,T]$.
    Let $\rho$ be the solution to \eqref{eq:ex-var-rho-ivp} with the vector field $u$ in \eqref{eq:ex-var-u} and initial value $p$. Define a functional $G: P \to \Rbb$ such that 
    \begin{equation}
        \label{eq:ex-var-G}
        G(q) = -\frac{1}{2} \int_{\Omega} |x|^{2} q(x) \, dx, \quad \forall \, q \in P .
    \end{equation}
    Notice that $\frac{\delta}{\delta q}G(q) = - \frac{1}{2}|x|^{2}$ for all $q$.
    
    We can verify that
    \begin{equation}
        \label{eq:ex-var-phi-sol}
        \phi_{t}(x) = \frac{|x|^{2}}{2(t-T-1)}
    \end{equation}
    is the solution to the terminal value problem \eqref{eq:var-phi-tvp} with $u$ in \eqref{eq:ex-var-u} and terminal value $\phi_{T} = \frac{\delta}{\delta \rho_{T}}G(\rho_{T}) = - \frac{1}{2}|x|^{2}$. 
    
    Now, from Proposition \ref{prop:var-g}, we claim that
    \begin{equation}
        \label{eq:ex-G-var}
        \frac{\delta}{\delta p} G(\rho_{T}) = \phi_{0} = - \frac{|x|^{2}}{2(T+1)} .
    \end{equation}

    To justify the claim \eqref{eq:ex-G-var}, we first notice that 
    \begin{equation*}
        \label{eq:ex-var-rho}
        \rho_{t}(x) = \Bigl(\frac{T+1}{T+1-t}\Bigr)^{d/2} p\biggl(\Bigl(\frac{T+1}{T+1-t}\Bigr)^{1/2}x\biggr) 
    \end{equation*}
    is the solution to the initial value problem \eqref{eq:ex-var-rho-ivp} with the vector field $u$ in \eqref{eq:ex-var-u} and the given initial value $p$. Therefore
    \begin{equation*}
        \label{eq:ex-var-rhoT}
        \rho_{T}(x) = (T+1)^{d/2} p\bigl( \sqrt{T+1} x \bigr) .
    \end{equation*}
    Hence we obtain
    \begin{align*}
        G(\rho_{T})
        & = - \frac{1}{2} \int_{\Omega} |x|^{2} \rho_{T}(x) \, dx \\
        & = - \frac{1}{2} \int_{\Omega} |x|^{2} (T+1)^{d/2} p\bigl( \sqrt{T+1} x \bigr) \, dx \\
        & = - \frac{1}{2(T+1)} \int_{\Omega} |y|^{2} p(y) \, dy ,
    \end{align*}
    where the last equality is due to the change of variables $y = \sqrt{T+1}x$. Therefore $\frac{\delta}{\delta p} G(\rho_{T}) = - \frac{|x|^{2}}{2(T+1)}$, which justifies \eqref{eq:ex-G-var}.
\end{example}

\section{Algorithmic Development}
\label{sec:algorithm}

In this section, we develop a numerical algorithm in light of the maximum principle (Theorem \ref{thm:maximum-principle}) to solve the probability control problem \eqref{eq:control-problem}. 
We provide a general description of this algorithm in Section \ref{subsec:algorithm} and some implementation details in Section \ref{subsec:alg-details}. 
%
% For the curious reader, a more comprehensive convergence analysis of the algorithm is available on arXiv, but is excluded here for brevity and clarity.

\subsection{Description of the Proposed Algorithm}
\label{subsec:algorithm}

The development of our algorithm was inspired by \cite{bonnans1986algorithm,sakawa1980global}, which are designed for classical optimal control problems. 
The idea is alternately updating $(\rho,u)$ and $\phi$. 
Namely, we begin with an initial guess $u^{0}$, and find its corresponding $\rho^{0}$,
Let $k=1,2,\dots$ be the iteration number.
In the $k$th iteration, we compute $\phi^{k}$ that satisfies the adjoint PDE \eqref{eq:adj} and terminal condition using $(\rho^{k-1},u^{k-1})$ obtained in the previous iteration; compute $(\rho^{k},u^{k})$ based on the maximum principle and the continuity equation using $\phi^{k}$; and then move on to the next iteration.
We provide more details about this alternating update scheme in the Algorithm Iterations (Section \ref{subsubsec:alg-iter}) below.

In our experiments, we are interested in solving optimal probability control problem in high dimensions, e.g., $d \ge 10$. To this end, we parameterize $u: \Rbb^{d} \times [0,T] \to \Rbb^{d}$ and $\phi: \Rbb^{d} \times [0,T] \to \Rbb$ using reduced-order models.
More specifically, we parameterize them as deep neural networks (DNNs).
DNNs are effective function approximators by using finite amounts of network parameters.
With DNN parameterizations, our algorithm does not require any spatial discretization as in traditional numerical PDE solvers, such as finite difference and finite element methods, which suffer the curse of dimensionality.

Since we use DNNs to parameterize $u$ and $\phi$, we need to update their network parameters in our numerical algorithm.
For notation simplicity, we will also use $u$ to denote its network parameters hereafter. For example, by solving for $u^k$ in the $k$th iteration of our algorithm, we mean finding the optimal network parameters of $u^{k}$. The same for $\phik$.

In our experiments, we artificially simulate a swarm of $N$ agents, denoted by $\{x_i(t) \in \Omega: i = 1,\dots , N\}$, to represent the probability density function $\rho_{t}$. In other words, $\{x_{i}(t)\}$ are $N$ samples following the probability $\rho_{t}$. 
Then the continuity equation \eqref{eq:ct-eq} can be equivalently written as $N$ trajectories solving the ODEs \eqref{eq:agent-ode} with initials $x_i(0) \sim p$ with the current guess $u$. Our algorithm will track these ODEs jointly with $u$ by the Neural ODE method \cite{chen2018neural}. 
%
% Thus, the probability density $\rho_t$ satisfies the continuity equation with the $u$ being used.

Our algorithm is summarized in Algorithm \ref{alg:density-control}. Next, we provide a description of the initialization, the iterations, and the termination criterion of this algorithm.
%\vspace{6pt}

\begin{algorithm}
    \caption{Numerical method for solving \eqref{eq:control-problem}}
    \label{alg:density-control}
    \begin{algorithmic}[1]
        \Require{Initial guess $u^0$, sample size $N$, $k=1$, max iteration number $K$, and tolerance $\epsilon_{\text{tol}}>0$.}
        \State{Generate i.i.d.\ samples $x_i(0) \sim p$ for $i = 1,\dots, N$.}
        \State{Solve the ODEs $x_i'(t) = u^0_{t}(x_i(t))$ to represent $\rho_t^0$.}
        \While{$k \le K$} 
        \State{Set $\phik$ to the minimizer of \eqref{eq:solve-phi}.}
        \State{Set $(\rho^{k},u^{k})$ to the minimizer \eqref{eq:solve-u}, where $\rho_t^{k}$ is represented as $\{x_i^{k}(t): i=1,\dots,N\}$.}
        \If{$\int_{0}^{T}\|\utkp - \utk \|_2^2 \,dt < \epsilon_{\text{tol}}$} 
            \State Break.
        \Else
            \State $k \gets k+1$.
        \EndIf
        \EndWhile
        \Ensure{Optimal control vector field $u^k$.}
    \end{algorithmic}
\end{algorithm}

\subsubsection{Initialization} 

We first determine the DNN architectures of $\phi$ and $u$. By network architecture, we meant the general structure, such as type, depth, width, and activation functions, of the network. The choice of such architecture is generally flexible. 
As $u: \Rbb^{d} \times [0,T] \to \Rbb^{d}$, we can parameterize it as a small neural network with input $(t,x)$, which is of dimension $d+1$, and output $u(t,x)$, which is of dimension $d$. 
The same for $\phi$ but the output dimension is 1.
To ensure $u$ is Lipschitz continuous, we use smooth and globally Lipschitz activation functions, and restrict the parameter values within a bounded set. 
%
% Then both DNNs $\phi$ and $u$ are Lipschitz continuous functions and smooth, or more precisely, with uniformly bounded second-order derivatives, on $[0,T]\times {\Omega}$.

We can initialize $u^0$ randomly using common network initialization method. This can be improved if one can provide an initial guess vector field with explicit formula that is close to the true $u(t,x)$, and train $u^0$ to approximate this initial guess field as the actual initial.
Then we generate samples $\{x_i(t): i=1,\dots,N,\, 0\le t\le T\}$ by solving the ODEs \eqref{eq:agent-ode} as described above using $u^0$, and these samples represent $\rho^0$. Now we have an initial $(\rho^0,u^0)$.
We also select a maximum iteration number $K \in \Nbb$ and a tolerance $\epsilon_{\text{tol}} > 0$ for termination criterion.

%\vspace{6pt}
\subsubsection{Algorithm Iterations} 
\label{subsubsec:alg-iter}

Let $k=1,\dots,K$ be the iteration number. In the $k$th iteration, we have $(\rho^{k-1},u^{k-1})$ from the previous iteration (or initial guess), and compute $\phik$ and $(\rho^{k},u^{k})$ in order as below.

To compute $\phik$, we need to solve the adjoint PDE \eqref{eq:adj} and the corresponding terminal condition using $(\rho^{k-1},u^{k-1})$. 
To this end, we apply the method of physics-informed neural network (PINN) \cite{raissi2019physics-informed} by formulating the loss function of $\phik$ as the squared error in the adjoint PDE subject to the terminal condition.
Namely, we set $\phik$ to be the minimizer of the following loss function
\begin{equation}
    \label{eq:solve-phi}
    \int_0^T\int_{\Omega} \Big|\pt \phi_t + u_t^k \cdot \nabla \phi_t + \frac{\delta}{\delta \rho_{t}^{k-1}} R(\rho_t^{k-1}, u_t^{k-1})\Big|^2 \,dx dt
\end{equation}
subject to $\phi^{k}_T=\frac{\delta}{\delta \rho_{T}^{k-1}} G(\rho^{k-1}_T)$. 
%
% As we can see, the loss function in \eqref{eq:solve-phi} is the squared error of the adjoint equation \eqref{eq:adj} on $[0,T]\times \Omega$, and the constraint is due to the terminal condition.
%
The PINN method approximates the integral like \eqref{eq:solve-phi} by Monte Carlo integration. See the implementation details of this approximation for \eqref{eq:solve-phi} in Section \ref{subsec:alg-details} below. The constraint can be relaxed by replacing it with a penalty term $\mu \|\phi_T - \frac{\delta}{\delta \rho_{T}^{k-1}} G(\rho_T^{k-1})\|_2^2$, where $\mu>0$ is the penalty weight parameter, added to the loss function \eqref{eq:solve-phi}. 
Alternatively, we can use a simple trick of network architecture set up such that this constraint is automatically satisfied, as we show in Section \ref{sec:results} below.

Besides PINN, there are other methods that can be adopted to solve \eqref{eq:solve-phi}. In particular, the adjoint PDE \eqref{eq:adj} is a first-order linear evolution PDE backward in time with a terminal value, it can be solved by converting the problem to solving an ODE of the network parameters backward in time \cite{du2021evolutional}.
In the present work, we use PINN for simplicity.

Once $\phik$ is obtained, we invoke the maximum principle \eqref{eq:pmp} and set $(\rho^{k},u^{k})$ as the minimizer of the following loss function with proximal point regularization:
\begin{equation}
    \label{eq:solve-u}
    -\int_0^T H(\rho_t, \phi_t^{k}, u_t) \,dt + \frac{1}{2\varepsilon_k} \int_0^T \|u_t - \utk\|_2^2 \, dt ,
\end{equation}
subject to $\pt \rho_t + \nabla \cdot (\rho_t u_t)=0$ and $\rho_0 = p$.
Here $\varepsilon_{k}>0$ is the step size. 
We use the Neural ODE method \cite{chen2018neural} to find the gradient of the loss function with respect to the parameters of $u$. To use this method, we can artificially simulate $N$ agents $\{x_{i}(t): i = 1,\dots,N, \, 0 \le t \le T \}$ to represent $\rho^{k}$ and finds the optimal $u$ jointly. Note that this ensures that the continuity equation and initial condition of $\rho^{k}$ are automatically satisfied.

%\vspace{6pt}

\subsubsection{Termination Criterion} We terminate the algorithm if $k$ reaches the maximum iteration number $K$ or  $\int_{0}^{T}\|\utkp - \utk \|_2^2 \,dt < \epsilon_{\text{tol}}$ as set in the initialization step.

%\vspace{6pt}

There are some additional problem-specific details about solving the minimization subproblems in practical implementation. We will provide them in Section \ref{subsec:alg-details}. 
We now show that Algorithm \ref{alg:density-control} converges under certain conditions.

\subsection{Convergence Analysis of the Proposed Algorithm}
\label{subsec:alg-convergence}

We consider the convergence of Algorithm \ref{alg:density-control} in solving the density control problem \eqref{eq:control-problem} in this subsection. 
Specially, we will prove several convergence properties of the sequence $\{u^k\}$ generated by Algorithm \ref{alg:density-control} from any initial $u^0$.
The convergence proofs require several additional assumptions, which we will highlight in this subsection.
\begin{assumption}
    \label{assumption:rho-bounded}
    There exists $M_{\rho}>0$ such that, for all $u \in U_T$, the corresponding density $\rho$ started from initial $p$ and controlled by $u$ satisfies $\|\nabla\rho_t\|_{\infty} \leq M_{\rho}$ for all $t \in [0,T]$.
\end{assumption}
Assumption \ref{assumption:rho-bounded} is valid if $\| \nabla p \|_{\infty} = \sup_{x \in {\Omega}}|\nabla p(x)|$ is finite and we also impose a uniform upper bound on all $u$ in $U_{T}$, with a fixed finite $T$ in consideration. 
%
% The smoothness conditions of the initial $p$ and $U_T$, together with finite terminal time $T$, can often result in a uniform $L^{\infty}$ bound $M_{\rho}>0$ of $\nabla \rho_t$ for all $t$ for the corresponding density $\rho \in P_T$. 
%
In real-world applications, the agents (such as robots) representing $\rho_t$ along time $t$ maintain certain distances from each other to prevent collisions, as considered in our experiments below. In this case, penalty on large $\nabla \rho_{t}$ is implicitly imposed and Assumption \ref{assumption:rho-bounded} holds.

\begin{lemma}
    \label{lemma:pq-bd}
    Suppose Assumption \ref{assumption:rho-bounded} holds. Let $u, v \in U_T$ and $\rho, \pi \in P_T$, where $\rho$ and $\pi$ are independently driven by $u$ and $v$, i.e., satisfy the continuity equations with time-dependent vector fields $u$ and $v$, from two initial values $p \in P$ and $q \in P$, respectively. Define $E(h) := \frac{1}{2} \|\rho_h - \pi_h \|_2^2$ for all $h \ge 0$. Then 
    \begin{equation}
        \label{eq:pq-bd}
        \dot{E}(0) \le \sqrt{2} M_{\rho} \sqrt{E(0)} \|u_0 - v_0 \|_2 + d M_U E(0).
    \end{equation}
\end{lemma}

\begin{proof}
    Let $\xh(z)$ be the push-forward of $z \in \Omega$ by $u$ up to time $h$, i.e., $\xh(z)$ is the solution $x(t;z)$ of the initial value problem
    \begin{equation*}
        \begin{cases}
            \frac{d}{dt} x(t;z) = u_t(x(t;z)), \quad \forall\, t \in [0,h] \\
            x(0;z) = z,
        \end{cases}
    \end{equation*}
    at time $t=h$, namely $\xh(z) = x(h;z)$. Let $\yh(z)$ be defined in the same way but with $v$ instead.

    For all $h > 0$ sufficiently small, we know there exist $\utilde,\vtilde : \Omega \times [0,h] \to \Rbb^{d}$ with 
    \[
    \|\nabla \utilde \|_{\infty}, \ \  \|\nabla \vtilde \|_{\infty}, \ \ \|\partial_h \utilde \|_{\infty}, \ \ \|\partial_h \vtilde \|_{\infty} \ \le \ M_{U,h}
    \]
    on $\Omega \times [0,h]$ for some $M_{U,h} > 0$, such that the $\xh$ and $\yh$ can be represented (e.g., by Taylor expansion) as
    \begin{subequations}
    \begin{align}
        \xh(z) & = z + u_0(z) h + \utilde_0(z,h) h^2 \label{eq:xh}, \\
        \yh(z) & = z + v_0(z) h + \vtilde_0(z,h) h^2 \label{eq:yh},
    \end{align}    
    \end{subequations}
    and they are both bijective 
    % (in fact diffeomorphism in our numerical setting) 
    on $\Omega$, and $\det(\nabla_z \xh(z))$, $\det(\nabla_z \yh(z)) > 0$ for all $z \in \Omega$. 
    We used, for example, $u_{t}(\cdot)$ to represent $u(\cdot,t)$ for short in \eqref{eq:xh} and \eqref{eq:yh}. In addition, notice that all of $x_0$, $y_0$, $x_0^{-1}$, and $y_0^{-1}$ are the identity map.
    Since
    \begin{align}
        E(h)
        & = \frac{1}{2} \int_{\Omega} |\rho_h(\xh(z)) - \pi_h(\xh(z))|^2 \,d \xh(z) \nonumber \\
        & = \frac{1}{2} \int_{\Omega} |\rho_0(z) - \pi_h(\yh(\yhinv(\xh(z))))|^2 \det(\nabla_z \xh(z))\,dz \label{eq:Ecal} \\
        & = \frac{1}{2} \int_{\Omega} |\rho_0(z) - \pi_0(\yhinv(\xh(z)))|^2 \det(\nabla_z \xh(z))\,dz, \nonumber
    \end{align}
    where the second and third equalities are due to the fact $\rho_h(x_h(w))=\rho_0(w)$ and $\pi_h(y_h(w))=\pi_0(w)$ for any $w \in \Omega$ since $\xh$ and $\yh$, as well as $\xhinv$ and $\yhinv$, are bijective. 
    Thus, we have
    \begin{align}
        \dot{E}(h) 
        & = - \int_{\Omega} (\rho_0(z) - \pi_0(\yhinv(\xh(z))) \ \nabla \pi_0(\yhinv ( \xh(z) )) \cdot J_1(z,h) \ \det(\nabla_z \xh(z)) \,dz \nonumber \\
        & \qquad + \frac{1}{2} \int_{\Omega} |\rho_0(z) - \pi_0(\yhinv(\xh(z)))|^2 J_2(z,h) \,dz \label{eq:dEcal}
    \end{align}
    where $J_1$ and $J_2$ are defined as
    \begin{subequations}
        \begin{align}
        J_1(z,h) & :=\frac{d}{dh}(\yhinv(\xh(z))) \in \Rbb^{d}, \label{eq:I1} \\
        J_2(z,h) & :=\frac{d}{dh}(\det(\xh(z))) \in \Rbb. \label{eq:I2}
        \end{align}
    \end{subequations}
    Note that we have
    \begin{equation}
    \label{eq:dxh-dh}
        \frac{d}{dh}\xh(z) = \frac{d}{dh}( \yh ( \yhinv(\xh(z)))) = \partial_h \yh ( \yhinv(\xh(z))) + \nabla \yh ( \yhinv(\xh(z))) J_1(z,h)\ .
    \end{equation}
    By taking $h=0$, we see the left-hand side of \eqref{eq:dxh-dh} is
    \begin{equation}
    \label{eq:dxhdh-lhs}
        \frac{d}{dh} \xh(z) \Big|_{h=0} = u_0(z) 
    \end{equation}
    and the right-hand side of \eqref{eq:dxh-dh} is
    \begin{align}
        & \quad \ \partial_h \yh ( \yhinv(\xh(z)))|_{h=0} + \nabla \yh ( \yhinv(\xh(z)))|_{h=0}\ J_1(z,h)|_{h=0} \nonumber \\
        & = [v_0(\yhinv(\xh(z))) + 2h \tilde{v}_0(\yhinv(\xh(z)),h) + \partial_h \vtilde_0(\yhinv(\xh(z)),h) h^2 ]|_{h=0} \nonumber \\
        & \qquad \qquad + [I_d + \nabla v_0(\yhinv(\xh(z))) h + \nabla \vtilde_0(\yhinv(\xh(z)),h) h^2] |_{h=0}\ J_1(z,h)|_{h=0} \label{eq:dxhdh-rhs} \\
        & = v_0(z) + J_1(z,0). \nonumber 
    \end{align}
    where $I_d$ is the $d\times d$ identity matrix.
    Plugging \eqref{eq:dxhdh-lhs} and \eqref{eq:dxhdh-rhs} into \eqref{eq:dxh-dh}, we obtain
    \begin{equation}
    \label{eq:I1z0}
        J_1(z,0) = u_0(z) - v_0(z) .
    \end{equation}

    Furthermore, we know by the Jacobi formula that for any matrix-valued function $A: \Rbb \to \Rbb^{d\times d}$ which is invertible at $h$, there is
    \[
    \frac{d}{dh} \det(A(h)) = \det(A(h)) \ \mathrm{tr} \del[2]{A(h)^{-1} \frac{d}{dh}A(h) }.
    \]
    Therefore, for any $z \in \Omega$, there are
    \[
    \nabla_z \xh(z)|_{h=0} = (I_{d} + \nabla u_0(z) h + \nabla \tilde{u}_0(z,h) h^2)|_{h=0} = I_{d},
    \]
    and
    \[
    \frac{d}{dh} \nabla_z \xh(z) \Big\vert_{h=0} = \frac{d}{dh}(I_{d} + \nabla u_0(z) h + \nabla \tilde{u}_0(z,h) h^2) \Big|_{h=0} = \nabla u_0(z),
    \]
    and hence
    \begin{equation}
    \label{eq:I2z0}
        J_2(z,0) = \det(\nabla_z \xh(z))|_{h=0} \ \mathrm{tr} \del[2]{\nabla_z \xh(z)^{-1}|_{h=0} \frac{d}{dh}\nabla_z \xh(z) \Big|_{h=0} } = \mathrm{tr}(\nabla u_0(z)) = \nabla \cdot u_0(z)\ .
    \end{equation}
    % where $\nabla \cdot u_0 (z) = \sum_{i=1}^{d} \partial_{z_i}u_0(z)$ for $z = (z_1,\dots,z_d)$.
    Therefore, by taking $h=0$ and \eqref{eq:I1z0} and \eqref{eq:I2z0}, we obtain from \eqref{eq:dEcal} that
    \begin{align*}
        \dot{E}(0)
        &= - \int_{\Omega} (\rho_0(z) - \pi_0(z)) \nabla \pi_0(z) \cdot (u_0(z) - v_0(z)) \,dz + \frac{1}{2} \int_{\Omega} |\rho_0(z) - \pi_0(z)|^2\, (\nabla \cdot u_0(z)) \,dz \\
        & \le M_{\rho} \|\rho_0 - \pi_0 \|_2 \|u_0 - v_0\|_2 + \frac{d M_U}{2} \|\rho_0 - \pi_0 \|_2^2\ ,
    \end{align*}
    where we used Assumption \ref{assumption:rho-bounded} to obtain $\|\nabla \pi_0\|_\infty \le M_{\rho}$ in the inequality. Therefore, \eqref{eq:pq-bd} is justified.
\end{proof}

\begin{proposition}
    \label{prop:alpha-beta}
    Consider the same setting as in Lemma \ref{lemma:pq-bd} but with an identical initial value $\rho_0=\pi_0$. Then for any $t \in [0,T]$, there is 
    \begin{equation}
        \label{eq:inst-bound}
        \frac{d}{dt} \Big( \frac{1}{2} \| \rho_t - \pi_t \|_2^2 \Big) \le \sqrt{2} M_{\rho} \| u_t - v_t \|_2 \Big( \frac{1}{2} \| \rho_t - \pi_t \|_2^2 \Big)^{1/2} + \frac{d M_U}{2} \|\rho_t - \pi_t \|_2^2 \ .
    \end{equation}
    Furthermore, there is
    \begin{equation}
        \label{eq:total-bound}
        \frac{1}{2} \| \rho_t - \pi_t \|_2^2 \le M_\rho^2 t e^{d M_U t} \int_{0}^{t} \frac{1}{2}\|u_s - v_s \|_2^2 \, ds \ .
    \end{equation}
\end{proposition}

\begin{proof}
    The bound \eqref{eq:inst-bound} is directly implied by Lemma \ref{lemma:pq-bd} with $0$ replaced by any arbitrary $t \in [0,T]$. To show \eqref{eq:total-bound}, we define $\delta(t):= \frac{1}{2} \|\rho_t - \pi_t \|_2^2$. Then by \eqref{eq:inst-bound} we have
    \begin{equation}
        \label{eq:ddelta}
        \dot{\delta}(t) = d M_U \delta(t) + \sqrt{2} M_{\rho} \|u_t - v_t\|_2 \sqrt{\delta(t)}\ .
    \end{equation}
    Dividing both sides by $2 \sqrt{\delta(t)}$, we obtain
    \begin{equation*}
        \frac{\dot{\delta}(t)}{2 \sqrt{\delta(t)}} = \frac{d}{dt} \sqrt{\delta(t)}  = \frac{d M_U}{2} \sqrt{\delta(t)} + \frac{M_{\rho}}{\sqrt{2}} \|u_t - v_t\|_2\ .
    \end{equation*}
    By the Gr\"{o}nwall inequality, we have
    \begin{align*}
        \sqrt{\delta(t)} 
        & \le \Big( \sqrt{\delta(0)} + \frac{M_{\rho}}{\sqrt{2}} \int_{0}^{t} \|u_s - v_s\|_2 e^{-d M_U s / 2} \, ds \Big) e^{d M_U t /2} \le \frac{1}{\sqrt{2}} M_{\rho} e^{d M_U t/2} \int_0^t \|u_s - v_s\|_2\,ds\ ,
    \end{align*}
    where $\delta(0)=0$ since $\rho_0 = \pi_0$.
    Therefore, we obtain
    \begin{equation*}
        \frac{1}{2}\|\rho_t - \pi_t\|_2^2 = \delta(t)  \le \frac{1}{2} M_{\rho}^2 e^{d M_U t} \Big( \int_0^t \|u_s - v_s\|_2 \, ds \Big)^2 \le \frac{1}{2} M_\rho^2t e^{d M_U t} \int_0^t \|u_s - v_s\|_2^2 \, ds
    \end{equation*}
    where the last inequality is due to the Cauchy--Schwarz inequality. This justifies \eqref{eq:total-bound}.
\end{proof}

\begin{assumption}\label{assump:alg}
    We make the following assumptions:
    \begin{enumerate}[label=(\roman*)]
        \item\label{assump:alg-lipschitz} $\frac{\delta}{\delta \rho} H$ and $\frac{\delta}{\delta \rho} G$ are $L_{\rho}$- and $L_G$-Lipshitz in $\rho$, respectively, and $\frac{\delta}{\delta u} H$ is $L_u$-Lipschitz in $u$, for some $L_{\rho}, L_{G}, L_{u} > 0 $.
        \item\label{assump:alg-H-concave} $H$ is concave in $u$.
    \end{enumerate}
\end{assumption}

% We also remark that the parameterization of $u$ as a deep neural network $u_{\theta}$ with parameter $\theta$ can ensure $u \in U_T$, $\|\nabla^2 u\|_{L^{\infty}(U_T)} \le M_{U_T}$, and $u$ to be uniformly $L_{B_U}$-Lipschitz on $U_T$ for some $L_{B_U}>0$, when the architecture of $u$ is fixed with Lipschitz activations and uniformly bounded parameter $\theta$. For notation simplicity, we omit $\theta$ and by solving $u$ we mean by solving for the network parameter $\theta$ of $u_{\theta}$.

% \begin{assumption}
%     \label{assump:interior}
%     We assume all local optimal solutions are in $\intr(U_T)$.
% \end{assumption}
% Assumption \ref{assump:interior} is valid in most practical scenarios for the following reasons. The bound $M_U$ to determine $U_T$ is manually determined and thus can be set very large such that no optimal solution can reach the boundary of $U_T$. This can be ensured if an artificial regularization term $\gamma \|u\|_{W^{1,\infty}(\Omega_T)}$ with a small $\gamma >0$ is added to the cost functional $I[u]$, because it prevents $u$ from approaching the boundary of $M_U$ by dominating other terms in $I[u]$ and raising the total cost. The benefit of this assumption is that in the proof of next theorem we only need to consider interior local optimal solutions so they are not affected by the boundedness constraint in $U_T$. The limitation is that it causes a theoretically tighter upper bound on the step sizes $\varepsilon_k$ in \eqref{eq:vareps-bound} in the proof.

\begin{theorem}
\label{thm:alg-convergence}
Suppose Assumptions \ref{assumption:rho-bounded} and \ref{assump:alg} hold. 
Let $c>0$ be arbitrary. Denote 
\begin{equation}
    \label{eq:stepsize-bound}
    \varepsilon_{\max} := \frac{2}{L_u + (L_\rho T + L_G) M_{\rho}^2 d M_U + c}
\end{equation}
and choose $\varepsilon_{\min} \in (0,\varepsilon_{\max})$ arbitrarily.
Let $\{(\rho^k,\phik,\uk):k \in \Nbb\}$ be a sequence generated by Algorithm \ref{alg:density-control} with $\varepsilon_k$ chosen from $[\varepsilon_{\min}, \varepsilon_{\max}]$ for every $k$ and any initial guess $u^0 \in U$. Then
\begin{enumerate}
    \item[(i)] $I[u^{k}]$ is non-decreasing. 
    \item[(ii)] If in addition $\Omega$ is bounded and $\{\uk\}$ is uniformly bounded, then $\{\uk\}$ has at least one uniformly convergent subsequence with some limit $\uhat \in U_{T}$. Moreover, $I[u^{k}] \to I[\uhat]$.
    % must be some local optimal solution of \eqref{eq:control-problem}, and all these local optimal solutions have the same reward functional value, denoted by $\hat{I}$.
    % \item The reward functional value $I[\uk]$ is non-decreasing and converges to $\hat{I}$ as $k\to\infty$.
\end{enumerate}
\end{theorem}

\begin{proof}
(i) For notation simplicity, we denote 
\begin{equation*}
    \alpha^k := \ukp-\uk, \qquad \beta^k := \rho^{k+1}-\rho^{k},
\end{equation*}
and
\begin{equation*}
    H_{t}^{k}:= H(\rho_{t}^{k}, \phi_{t}^{k}, u_{t}^{k}) = \langle \rho_{t}^{k}, u_{t}^{k} \cdot \nabla \phi_{t}^{k} \rangle + R(\rho_{t}^{k},u_{t}^{k}) .
\end{equation*}
Recall the definition of total reward functional $I$ in \eqref{eq:control-problem} and the Hamiltonian functional $H$ in \eqref{eq:Hamiltonian}, we have 
%
% \begin{equation}
%     \label{eq:I-Hamiltonian}
%     I[u] = \int_0^T \Big( - H(\rho_t, \phi_t, u_t) + \langle \rho_t, \ u_t \cdot \nabla \phi_t \rangle \Big) \, dt - G(\rho_{T}) \ .
% \end{equation}
% %
% Note that $I$ does not directly depend on $\phi$, since $\phi$ is canceled within the integrand due to the definition of $H$ in \eqref{eq:Hamiltonian}. This can also be seen from the definition of $I$ in \eqref{eq:control-problem} where $\phi$ is not involved. Now we have
\begin{equation*}
    I[\uk] = \int_{0}^{T} \Bigl( H_{t}^{k} - \langle \rho_{t}^{k}, u_{t}^{k} \cdot \nabla \phi_{t}^{k} \rangle \Big) \, dt + G(\rho_{T}^{k}) 
\end{equation*}
and
\begin{equation*}
    I[\ukp] = \int_{0}^{T} \Bigl( H(\rhotkp, \phitk, \utkp) - \langle \rhotkp, \utkp \cdot \nabla \phitk \rangle \Big)\, dt + G(\rho_{T}^{k+1}) .
\end{equation*}
Therefore, 
\begin{align}
    I[\uk] - I[\ukp]
    & = \int_{0}^T\Big[ H_{t}^{k} -H(\rhotkp,\phitk,\utkp) + \langle \rhotkp, \utkp \cdot \nabla \phitk \rangle - \langle \rhotk, \utk \cdot \nabla \phitk \rangle \Big] dt \nonumber \\
    & \qquad \qquad - G(\rho_{T}^{k+1}) + G(\rho_{T}^k) \label{eq:I-diff} \\
    & = \int_{0}^T\left[ H(\rhotk,\phitk,\utk) -H_{t}^{k} + \langle \pt \betatk, \phitk \rangle \right] \,dt - G(\rho_{T}^{k+1}) + G(\rho_{T}^k)  \ , \nonumber 
\end{align}
where we used integration by parts, the continuity equations of $\rho^k$ and $\rho^{k+1}$ with $\uk$ and $\ukp$ respectively, and the definition of $\beta^k$ to obtain the second equality.
With Assumption \ref{assump:alg} (i) on the Lipschitz continuity of $\frac{\delta H}{\delta u}$ and $\frac{\delta H}{\delta \rho}$ (and $-\frac{\delta H}{\delta \rho}$ as well), we have
\begin{equation}
    H (\rhotkp, \phitk, \utk)
    \le H(\rhotkp, \phitk, \utkp)  - \Bigl\langle \frac{\delta}{\delta \utkp} H(\rhotkp, \phitk, \utkp), \alphatk \Bigr\rangle + \frac{L_u}{2} \|\alphatk \|_2^2  \label{eq:H-u-Lip}
\end{equation}
and
\begin{equation}
    -H(\rhotkp, \phitk, \utk) 
    \le - H_{t}^{k} - \Bigl\langle \frac{\delta}{\delta \rho_t^k} H_t^k, \betatk \Bigr\rangle + \frac{L_\rho}{2} \|\betatk\|_2^2\  . \label{eq:H-rho-Lip} 
\end{equation}
Combining \eqref{eq:H-u-Lip} and \eqref{eq:H-rho-Lip} yields
\begin{equation}
    \label{eq:H-Lip-bound}
    H_{t}^{k} -H(\rhotkp, \phitk, \utkp) \le - \Bigl\langle \frac{\delta}{\delta \utkp} H(\rhotkp, \phitk, \utkp), \alphatk \Bigr\rangle - \Bigl\langle \frac{\delta}{\delta \rho_t^k} H_t^k, \betatk \Bigr\rangle + \frac{L_u}{2} \|\alphatk \|_2^2+ \frac{L_\rho}{2} \|\betatk\|_2^2 .
\end{equation}
With Assumption \ref{assump:alg} (ii), the convexity of $U_T$ and the optimality condition of $u^{k+1}$ in the update step of $(\rho^{k+1}, \ukp)$ in Algorithm \ref{alg:density-control}, we have
\begin{align}
    \Bigl\langle -\frac{\delta}{\delta \utkp} H(\rhotkp,\phitk,\utkp) + \frac{1}{\varepsilon_k} \alphatk, \ \alphatk \Bigr\rangle  \le 0 . \label{eq:proof-H2}
\end{align}
%
% We notice that \eqref{eq:proof-H2} yields
% \begin{equation}
%     \label{eq:-epsalpha}
%     \Bigl\langle - \frac{\delta}{\delta \utkp} H(\rhotkp, \phitk, \utkp),\ \alphatk \Bigr\rangle \le - \frac{1}{\varepsilon_k} \|\alphatk \|_2^2 \ .
% \end{equation}
%
Plugging \eqref{eq:H-Lip-bound}, \eqref{eq:proof-H2}, and 
$\pt \phitk = -\frac{\delta}{\delta \rho_t^k} H_t^k$ (Proposition \ref{prop:control-Hamiltonian}) into \eqref{eq:I-diff},
we obtain
\begin{align} 
    I[\uk ] - I[\ukp] 
    & \le \int_0^T \Big( -\frac{1}{\varepsilon_k} \| \alphatk \|_2^2 + \frac{L_u}{2} \|\alphatk \|_2^2 + \frac{L_\rho}{2} \|\betatk\|_2^2 \Big) \, dt \label{eq:I-diff-tmp}\\
    & \qquad +\int_0^T \Bigl(\langle \pt \phitk,\ \betatk \rangle + \langle \phitk , \ \pt \betatk \rangle \Bigr)\, dt - G(\rho_{T}^{k+1}) + G(\rho_{T}^k). \nonumber
\end{align}
Since $\rho_0^{k+1} = \rho_0^k = p$ which implies $\beta_{0}^{k}= 0$, we have
\begin{equation}
    \label{eq:beta-phi}
    \int_0^T \Big(\langle \pt \phitk,\ \betatk \rangle + \langle \phitk, \ \pt \betatk \rangle \Big) \, dt 
    = \langle \betatk, \ \phitk \rangle \Big\vert_0^{T} = \langle \beta_T^k, \ \phi_T^k \rangle 
    = \Bigl\langle \rho_T^{k+1} - \rho_T^{k}, \ \frac{\delta}{\delta \rho_T^k} G(\rho_T^k) \Bigr\rangle.
\end{equation}
With Assumption \ref{assump:alg} (i) on the Lipschitz condition of $\frac{\delta}{\delta \rho} G$ (and $-\frac{\delta}{\delta \rho}G$ as well), we have
\begin{equation}
    \label{eq:G-Lip-bound}
    - G(\rho_T^{k+1}) + G(\rho_T^k) + \Bigl\langle \frac{\delta}{\delta \rho_{T}^k} G(\rho_T^k), \ \rho_T^{k+1} - \rho_T^k \Bigr\rangle \le \frac{L_G}{2} \|\rho_T^{k+1} - \rho_T^k\|_2^2  = \frac{L_G}{2} \|\beta_T^k\|_2^2\ .
\end{equation}
Applying \eqref{eq:I-diff-tmp}, \eqref{eq:beta-phi}, and \eqref{eq:G-Lip-bound}, we obtain
\begin{equation}
    \label{eq:I-diff-bound}
    I[\uk ] - I[\ukp] \le \int_0^T \Big( -\frac{1}{\varepsilon_k} \| \alphatk \|_2^2 + \frac{L_u}{2} \|\alphatk \|_2^2 + \frac{L_\rho}{2} \|\betatk\|_2^2 \Big) \, dt + \frac{L_G}{2} \| \beta_T^k \|_2^2 \ .
\end{equation}
Due to the estimate \eqref{eq:total-bound} in Proposition \ref{prop:alpha-beta}, we have
\[
\frac{1}{2} \| \betatk \|_2^2 \le M_{\rho}^2 d M_U t \int_0^{t}\frac{1}{2}\|\alpha_s^k\|_2^2 \,ds \le M_{\rho}^2 d M_U T \int_0^{T}\frac{1}{2}\|\alpha_t^k\|_2^2 \,dt
\]
for any $t \in [0, T]$, which further implies that
\begin{align}
    \frac{L_\rho}{2} \int_0^T \|\betatk\|_2^2 \,dt 
    & \le L_{\rho} M_{\rho}^2 d M_U T^2 \int_0^T \frac{1}{2} \|\alphatk\|_2^2 \,dt \ , \label{eq:int-beta-t} \\
    \frac{L_{G}}{2} \|\beta_T^k\|_2^2  
    & \le L_G M_{\rho}^2 d M_U T \int_0^T \frac{1}{2} \|\alphatk\|_2^2 \,dt \ . \label{eq:int-beta-T}
\end{align}
Applying \eqref{eq:int-beta-t} and \eqref{eq:int-beta-T} into \eqref{eq:I-diff-bound}, we obtain
\begin{equation}
    \label{eq:I-diff-bound-clean}
    I[\uk] - I[\ukp] \le \Big( -\frac{2}{\varepsilon_k} + L_u + (L_\rho T + L_G) M_{\rho}^2 d M_U \Big) \int_0^T \frac{1}{2} \|\alphatk\|_2^2 \, dt
\end{equation}
Therefore, 
% for any $c>0$ and a lower bound $\tilde{\varepsilon} \in (0, \frac{2}{L_u + (L_\rho T + L_G) M_{\rho}^2 d M_U + c})$, we can choose $\varepsilon_k$ such that
% \begin{equation}
%     \label{eq:vareps-bound}
%     0 < \tilde{\varepsilon} \le  \varepsilon_k \le \frac{2}{L_u + (L_\rho T + L_G) M_{\rho}^2 d M_U + c}
% \end{equation}
with $\varepsilon_{k} \in [\varepsilon_{\min}, \varepsilon_{\max}]$
for every $k$, we have $-\frac{2}{\varepsilon_k} + L_u + (L_\rho T + L_G) M_{\rho}^2 d M_U \le -c < 0$ and thus from \eqref{eq:I-diff-bound-clean} there is
\begin{equation}
    \label{eq:alpha-I-bound}
    c \int_{0}^{T} \frac{1}{2} \|\alphatk\|_2^2 \,dt \le I[\ukp] - I[\uk].
\end{equation}
This implies $I[\uk]$ is non-increasing.

(ii) By taking sum of \eqref{eq:alpha-I-bound} over $k=0,\dots,K-1$ for any $K \in \Nbb$, we obtain
\begin{equation*}
    c \sum_{k=0}^{K-1} \int_{0}^{T} \frac{1}{2} \|\alphatk\|_2^2 \,dt \le I[u^K] - I[u^0] \le I^{*} - I[u^0] < \infty ,
\end{equation*}
since $I[u^{K}] \le I^{*} < \infty$ for all $K$.
Letting $K\to \infty$, we know the series on the left-hand side converges, and thus
\begin{equation*}
    \|\ukp - \uk \|_{L^2(\OmegaT)}^{2} = \int_{0}^{T} \|\alphatk\|_2^{2} \, dt\to 0.
\end{equation*}
This also implies that $\| \ukp_{t} - \uk_{t} \|_{L^{2}(\bar{\Omega})} \to 0$ as $k \to \infty$ in $[0,T]$ almost everywhere.

Since $\Omega \in \Rbb^{d}$ is bounded, we know $\OmegaT := \bar{\Omega}\times [0,T]$ is compact in $\Rbb^{d+1}$. Provided $\{u^k\} \subset U_{T}$ (hence $u^{k}$ is $M_{U}$-Lipschitz continuous for all $k$) and $\|u^{k}\|_{\infty}$ are uniformly bounded, we know by the Arzel\`{a}--Ascoli theorem that there exists a subsequence $\{u^{k_j}:j \in \Nbb\}$ of $\{\uk\}$ and $\hat{u} \in U_T$, such that $u^{k_j} \to \hat{u}$ uniformly. As $\|\ukp - \uk \|_{L^2(\OmegaT)} \to 0$, we know $u^{k_j-1} \to \hat{u}$ as $j \to \infty$ in $\OmegaT$ in the $L^2$ sense.

Let $\hat{\rho}$ be the solution to \eqref{eq:ct-eq} with vector field $\hat{u}$ and initial value $p$. Notice that $\hat{\rho} \in P_{T}$ as $\hat{u} \in U_{T}$. 
By Proposition \ref{prop:alpha-beta}, more specifically, \eqref{eq:total-bound},  we have for all $t \in [0,T]$ that
\begin{equation*}
    % \| \rho^{k_{j}-1} \to \hat{\rho} \|_{L^{2}(\OmegaT)}^{2} = 
    \| \rho_{t}^{{k_{j}-1}} - \hat{\rho}_{t}\|_{2}^{2} \le 2 M_{\rho}^{2}T e^{d M_{U} T} \| u^{k_{j}-1} - \hat{u} \|_{L^{2}(\OmegaT)}^{2}  \to 0 
\end{equation*}
as $j \to \infty$.
Taking integral of $t$ over $[0,T]$ on both sides above, we know $\rho^{k_{j}-1} \to \hat{\rho}$ as $j \to \infty$ in $\OmegaT$ in the $L^{2}$ sense.

Let $\hat{\phi}$ be the adjoint function of $(\hat{\rho}, \hat{u})$ as in Definition \ref{def:adj}. Notice that the adjoint PDE is a first-order linear PDE. Therefore $\phi^{k_{j}}$, which is the adjoint of $(\rho^{k_{j}-1}, u^{k_{j}-1})$, converges to $\hat{\phi}$ as $j \to \infty$ in $\OmegaT$ in the $L^{2}$ sense.

Next, we show that $(\rhohat, \phihat, \uhat)$ satisfies the Maximum Principle in Theorem \ref{thm:maximum-principle}. By the definitions of $\rhohat$ and $\phihat$, we know they satisfy the Hamiltonian dynamics \eqref{eq:hamiltonian-dynamics} with control vector field $\uhat$. 
Furthermore, by the optimality condition of $u^{k_{j}}$, the convexity of $U_{T}$, and Assumption \ref{assump:alg} (ii), we have for any $u \in U_{T}$ that
\begin{equation}
    0 \le \Bigl\langle - \frac{\delta}{\delta u_{t}^{k_{j}}} H(\rho_{t}^{k_{j}}, \phi_{t}^{k_{j}}, u_{t}^{k_{j}}) + \frac{1}{\varepsilon_{k_{j}}} (u_{t}^{k_{j}} - u_{t}^{k_{j}-1} ), u_{t} - u_{t}^{k_{j}} \Bigr\rangle 
\end{equation}
for all $t \in [0,T]$.
Taking $j \to \infty$, and noticing that $\varepsilon_{k_{j}} \ge \varepsilon_{\min} > 0$ and $\| u_{t}^{k_{j}} - u_{t}^{k_{j}-1} \|_{L^{2}(\Omega)} \to 0$ for all $t$, we obtain
\begin{equation*}
    0 \le \Bigl\langle -\frac{\delta}{\delta \uhat} H(\rhohat_{t}, \phihat_{t}, \uhat_{t}) , u_{t} - \uhat_{t} \Bigr\rangle .
\end{equation*}
As $H$ is concave in $u$, $U_{T}$ is a convex set, and $u_{t}, \uhat_{t} \in U$, we see $H(\rhohat_{t}, \phihat_{t}, \uhat_{t}) = \max_{u_{t} \in U} H(\rhohat_{t}, \phihat_{t}, u_{t})$.

Since $u^{k_{j}} \to \uhat$ uniformly, we have $I[u^{k_{j}}] \to I[\uhat]$. Combining with (i) and $\{u^{k_{j}}\}$ being a subsequence of $\{u^{k}\}$, we know $I[u^{k}] \to I[\uhat]$.
\end{proof}

\subsection{Implementation Details}
\label{subsec:alg-details}

% DNN approach has emerged as a powerful tool in recent years to approximate functions and solve PDEs, particularly in high dimensions. We will use this approach to tackle the subproblems \eqref{eq:solve-phi} and \eqref{eq:solve-u}. 
%
% A representative method of this class is the PINN \cite{raissi2019physics-informed}. The idea is to parameterize the function (such as the solution to a PDE) as a DNN with input $(t,x) \in [0,T] \times \Omega$ as the input of the DNN, approximate the integrals in the loss function by Monte Carlo integrations, and then employ some variant of stochastic gradient descent method builtin existing optimization toolbox of deep learning package to find the optimal network parameter. During the optimization process, the gradients of the loss with respect to the network parameter are computed through back-propagation.

In this subsection, we provide some details about our numerical implementations in minimizing \eqref{eq:solve-phi} and \eqref{eq:solve-u}.

For \eqref{eq:solve-phi}, we use the PINN method \cite{raissi2019physics-informed}. More specifically, we sample $N$ points $\{(t_i,x_i): i\in [N]\}$ uniformly from $[0,T] \times \Omega$, and approximate the \eqref{eq:solve-phi} by
\begin{align}
    \frac{T|\Omega|}{N}\sum_{i=1}^N & \Big|\pt \phi_{t_{i}}(x_i) + u_{t_{i}}^k(x_i)\cdot \nabla \phi_{t_{i}}(x_i)  + \frac{\delta}{\delta \rho_{t_{i}}^{k-1}} R(\rho_{t_{i}}^{k-1},\uk_{t_{i}})(x_i)\Big|^2\ . \label{eq:phi-approx-loss}
\end{align}
It is known that \eqref{eq:solve-phi} is an unbiased estimation of the loss function \eqref{eq:phi-approx-loss}, and the variance reduces at the rate of $O(N^{-1/2})$.

If $\Omega$ is unbounded, we can adopt the importance sampling method which samples $x_i$'s from a reference density $p_{\text{ref}}$ supported on $\Omega$ and can be evaluated (e.g., $p_{\text{ref}}$ is Gaussian if $\Omega = \Rbb^d$) and divide the $i$th term of the summation in the loss function by $p_{\text{ref}}(x_i)$ in \eqref{eq:phi-approx-loss}. 
%
% If $\Omega$ is bounded, then we can draw samples uniformly within it if no better sampling method exists.

For \eqref{eq:solve-u}, we recall the definition of $H$ in \eqref{eq:Hamiltonian} and approximate the loss function by
\begin{align}
    - \frac{T}{N_{\Omega}N_{T}}\sum_{j=1}^{N_{T}}\sum_{i=1}^{N_{\Omega}} u_{t_{j}}(x_{i}(t_{j}))\cdot \nabla \phi_{t_{j}}^{k}(x_{i}(t_{j})) -  \frac{T|\Omega|}{N}\sum_{i=1}^N  R(\rho_{\hat{t}_{i}},u_{\hat{t}_{i}})(\hat{x}_{i})
    +  \frac{T|\Omega|}{2 \varepsilon_{k} N}\sum_{i=1}^N |u_{\hat{t}_i}(\hat{x}_i) - u^{k}_{\hat{t}_i}(\hat{x}_i)|^2, \label{eq:u-approx-loss}
\end{align}
where $\{x_{i}(t_{j}): i=1,\dots,N_{\Omega},\ j = 1,\dots,N_{T}\}$ are the states of artificially simulated agents at sampled time points in $[0,T]$ such that $x_{i}(t_{j}) \sim \rho_{t_{j}}$, and $\{(\hat{t}_{i},\hat{x}_{i}): i = 1,\dots, N\}$ are $N$ points uniformly sampled from $[0,T] \times \Omega$.
Then we apply the Neural ODE method to solve for $(\rho^{k},u^{k})$ which minimizes \eqref{eq:u-approx-loss} while satisfying the continuity equation and initial condition, where $\rho_{t_{j}}^{k}$ is represented by the artificially simulated agents $\{x_{i}(t_{j}): i = 1,\dots, N_{\Omega} \}$ for all $t_{j}$.
% , and $\{t_{j}: j = 1,\dots, N_{T}\}$ are $N_{T}$ time points uniformly sampled from $[0,T]$.

There are various treatments to handle the terms in the loss functions \eqref{eq:phi-approx-loss} and \eqref{eq:u-approx-loss}. For example, if $R(\rho_{t},u_{t}) = -\frac{1}{2} \int_{\Omega} |u_{t}|^{2} \rho_{t} \, dx$, then it can be approximated by the Monte Carlo integration 
\[
-\frac{1}{2}\frac{T}{N_{\Omega}N_{T}}\sum_{j=1}^{N_{T}}|u_{t_{j}}(x_{i}(t_{j}))|^{2},
\]
instead of the first term in the second line of \eqref{eq:u-approx-loss}. 
In the case where $R(\rho_{t},u_{t})$ contains a term 
\[
\int_{\Omega}\int_{\Omega}a(x,y)\rho_{t}(y)\rho_{t}(x)\,dydx,
\]
due to agent interaction specified by $a(x,y)$, we can simplify the computation of $R(\rho_{t},u_{t})$ and $\frac{\delta}{\delta \rho_{t}} R(\rho_{t},u_{t})$ by using an augmented variable $z:=(x,y) \in \Rbb^{2d}$ and setting its probability density as $\bar{\rho}(z):=\rho(x)\rho(y)$ and the control on $z$ to $(u(x),u(y))$. This makes the approximation of $R(\rho_{t},u_{t})$ and $\frac{\delta}{\delta \rho_{t}} R(\rho_{t},u_{t})$ easier to implement in practice. 
We adopt these treatments throughout our experiments to simplify \eqref{eq:phi-approx-loss} and \eqref{eq:u-approx-loss}.
We will release the code containing these details to the public upon the acceptance of this paper.

% \gaby{In the case where we want interaction term in $R(t,x;\rho,u)$ such as $R(\rho,u)=\int_{\Omega}\int_{\Omega}|x-y|^2\rho(x),\rho(y) dxdy $ we can make finding $\frac{\delta}{\delta \rho}R$ easier by using an augmented variable $z$ as a stack of two $d$-dimensional variables $x$ and $y$ and $\rho(z)$ as two independent distributions such that $\rho(z)=\rho_1(x)\rho_1(y)$ subject to the control $u(z)=(u_1(x),u_1(y))^{\top}$. Then in our example $\frac{\delta}{\delta \rho}R(t,z;\rho,u) = |x-y|^2$ which would make the approximate loss function in \eqref{eq:phi-approx-loss} simpler to calculate.}

\section{Experimental settings and results}
\label{sec:results}

We conduct several numerical experiments to demonstrate the performance of Algorithm \ref{alg:density-control} in solving the optimal probability control problem \eqref{eq:control-problem}. In these experiments, we only consider the case where the terminal reward function is given by $G(q)= \int_{\Omega} g(x) q(x)\,dx$, and thus $\frac{\delta}{\delta q} G (q) = g$ for all $q \in P$. We will specify $g$ in each experiment below.

In our numerical implementation, we parameterize $u$ as a standard residual network (ResNet) with two hidden layers, each of width 100, and ReLU as the activation functions. 
We parameterize $\phi$ as
\begin{equation}
    \label{eq:phi-parameterization}
    \phi_{t}(x):= \Big( 1- \frac{t}{T} \Big)\varphi(t,x)+\frac{t}{T}g(x),
\end{equation}
where $\varphi$ is also a ResNet with two hidden layers and width 100, and tanh as activation functions. 
The setting \eqref{eq:phi-parameterization} ensures that the terminal condition $\phi_{T} = g$ is always satisfied, hence eliminating this constraint in implementation. 
%
% Strictly speaking, we should use smooth activation functions such that the network parameterizing $u$ is smooth, which is consistent with the smoothness requirement in our proofs on algorithm convergence. However, ReLU empirically also performs well and thus we just employ this widely used activation function in DNNs in our experiments.
%

In each iteration of Algorithm \ref{alg:density-control}, we update $\phik$ by applying the Adam optimizer \cite{kingma2015adam:} to minimize the empirical loss function \eqref{eq:phi-approx-loss} using a mini-batch of $N=2,048$ newly drawn sample points for each iteration, with the default setting of $\beta_1=0.9$ and $\beta_2=0.999$ for $3,000$ steps (except for $k=1$ when we run $10,000$ steps). We also adopt the Adam optimizer to update $(\rho^{k},u^{k})$ jointly as in \eqref{eq:u-approx-loss} with the same parameters for 20 steps, where in each step we apply the Neural ODE method to compute the gradient of the loss function with respect to the network parameters of $u$. For simplicity, we set the step size $\varepsilon_k=0.25$ for all $k$. Experiments are conducted using PyTorch on a Windows 10 machine with one NVIDIA RTX GeForce 2080 Super GPU. All experiments in this section took a maximum of 10 minutes to complete the process outlined in Algorithm \ref{alg:density-control}.

For demonstration purposes, we conduct numerical experiments on three synthetic test problems:
\begin{itemize}
    \item Test 1 with agent interactions and problem domain $\Omega$ of dimension $d=8$;
    \item Test 2 with a cylindrical obstacle and problem dimensions $d=30$ and $d=100$;
    \item Test 3 with agent interactions and a double-wedge obstacle in problem dimension $d=30$.
\end{itemize}
% the first one with an interaction term and problem dimension $d=8$ (Test 1), the second with a regional cylindrical obstacle in dimensions $d=30$ and $d=100$ (Test 2), and the third with both interaction term and a wedge obstacle in dimension $d=30$ (Test 3). 
%
Since there are no closed form solutions for these test problems, we cannot compute the numerical errors of our results. 
Instead, we provide some visual demonstrations of these results as shown in the figures below.
%
% Instead, we provide three evaluations of Algorithm \ref{alg:density-control}: (i) Visual illustrations of the density evolution using agent movements to show the effect of the output control $u^K$; (ii) The cost functional value $I[u^k]$ versus iteration $k=0,1,\dots,K$; and (iii) The value of Hamiltonian $H_t^K/|H_0^K|$ versus time $t$, where $H_t^K:=(\rho_t^K, \phi_t^K, u_t^K)$ for $t \in [0,T]$.

\subsection{Test 1 with Agent Interactions}

We consider an optimal probability control problem where the running reward functional contains a nonlinear agent interaction term. The problem domain is set to $\Omega = \Rbb^{d}$ with $d=8$, and the running reward functional is set to
\begin{equation}
\label{eq:r-collision}
    R(\rho_t,u_t) = - \frac{1}{2}\int_{\Omega} |u_t|^2\rho_t\, dx - \gamma \int_{\Omega} \int_{\Omega} \frac{\rho_t(y)\rho_t(x)}{a(x,y)}\,dydx,
\end{equation}
where $a(x,y):=0.1 + |y-x|^2$ is the agent interaction (repelling) term that imposes penalty if two agents are too close which may lead to collision. We set the terminal reward function $g(x) = -|x|^2/2$ and the terminal time $T=1$, which encourages the agents to gather at the target point $0$ at $T=1$. 
The initial density is set to $p(x) = \mathcal{N}(x; -2 \mathbf{1}_d, 0.5 I_d)$, which is a Gaussian with mean $-2\mathbf{1}_d$ and covariance $0.5 I_d$. Here $\mathbf{1}_d$ denotes the $d$-dimensional vector with all components equal to $1$.
We test Algorithm \ref{alg:density-control} to solve the optimal probability control problem \eqref{eq:control-problem} without and with the interaction term in \eqref{eq:r-collision}by setting $\gamma = 0$ and $\gamma = 5$, respectively.

The resulting agent distributions at times $t=0$, $0.25$, $0.5$, $0.75$, and $1$ are shown with different colors at different dimension sections in Figure \ref{fig:numerical-d8}. 
In the top row of Figure \ref{fig:numerical-d8}, we observe the expected agent movements along time and tendency to avoid collision when $\gamma =5$. 
In the bottom row of Figure \ref{fig:numerical-d8}, we turn off the interaction term by setting $\gamma=0$, and observe that the agents tend to clump together as they approach the target point $0$.

\begin{figure}[htb]
    \centering
    \includegraphics[scale=0.55]{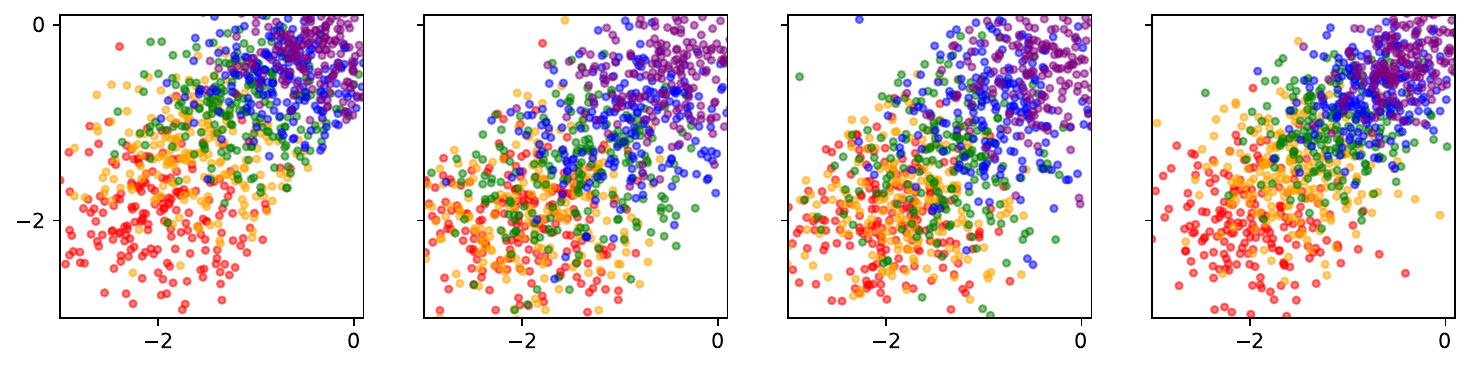}
    \includegraphics[scale=0.55]{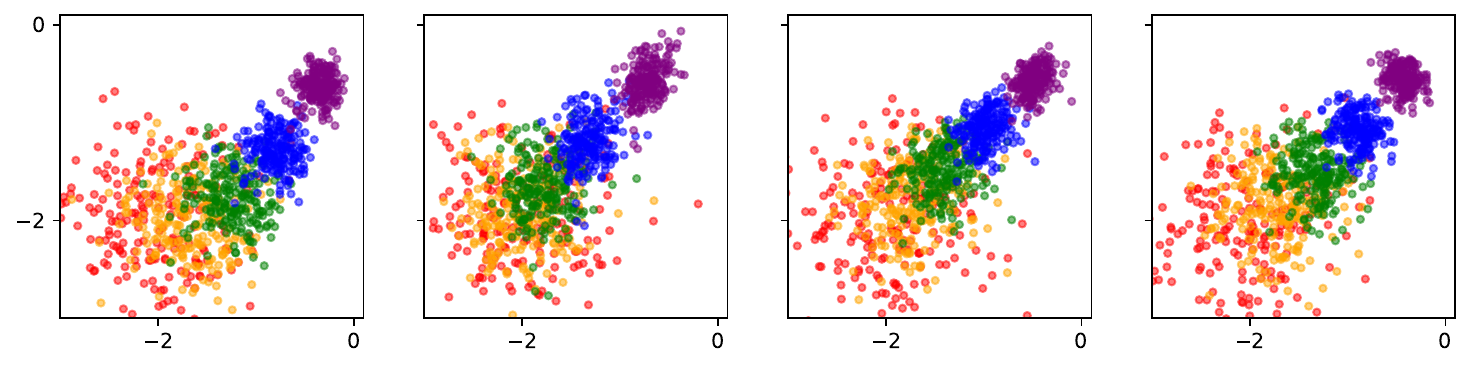}
    \caption{(Test 1) Visualization of agent movements under $u$ solved by Algorithm \ref{alg:density-control} for the probability control problem with running reward \eqref{eq:r-collision}. Agent distributions are show by colored circles at different times $t=0$ (red), $0.25$ (orange), $0.5$ (green), $0.75$ (blue), and $1$ (purple). Top row and bottom row show the results with and without agent collision preventions, i.e., $\gamma=5$ and $\gamma=0$, respectively. The plots from left to right columns show the agent distributions in the $(x_1,x_2)$, $(x_3,x_4)$, $(x_5,x_6)$, and $(x_7,x_8)$ coordinates, respectively.}
    \label{fig:numerical-d8}
\end{figure}

\subsection{Test 2 with Regional Obstacle}

We consider the probability control problem \eqref{eq:control-problem} where the goal is to move the agents close to a specified location $x^*$ without hitting on a fixed obstacle in $\Omega$. 
We set $\Omega = \Rbb^{d}$ with $d = 30$ and $d=100$, and in each case we place a cylindrical obstacle of radius $0.5$ with $\{(0,0,x_{3},\dots,x_{d}):x_{i} \in \Rbb,\, i=3,\dots,d\}$ as the center line. 
We define  
\begin{equation*}
    b(x)=
    \begin{cases}
        (x_1^2+x_2^2) - \frac{1}{4}, & \text{if} \ x_{1}^{2}+x_{2}^{2} > \frac{1}{4}, \\
        0, & \text{elsewhere} ,
    \end{cases}
\end{equation*}
which is 0 if $x$ is in the obstacle, and becomes larger as $x$ gets further away from the obstacle.
The running reward is set to
\begin{equation}
    \label{eq:R-cylindrical-obstacle}
    R(\rho_t,u_t)= - \frac{1}{2}\int|u_t|^2\rho_t \, dx - \int \frac{1}{\epsilon + b(x)}\rho_{t}(x)\,dx ,
\end{equation}
where $\epsilon=0.1$ is used for computation stability.

We set the terminal time $T=1$ and the terminal reward functional to
\begin{equation}
    \label{eq:G-cylindrical-obstacle}
    G(q)= - \int_{\Omega} |x - x^*|^2 q(x) \, dx
\end{equation}
where $x^*(1,-0.5,0,\dots,0) \in \Rbb^d$ is the target spot where the agents should get around at terminal time.

We set the initial probability density to
\begin{equation*}
    % \label{eq:p-init-cylindrical-obstacle}
    p(x) = \chi \Big( x_1 + \frac{1}{2} \Big) e^{x_1+\frac{1}{2}} \mathcal{N}(x_{2:d}\, ;\, (0.5,0,\dots,0),\, 0.25 I_{d-1})
\end{equation*}
where $x_{2:d}:=(x_2,\dots,x_d)$ and $\chi(z) = 1$ if $z\le 0$ and $0$ otherwise.  
Therefore, the agents are close to but not inside the obstacle at the initial time $t=0$.

We solve the optimal probability control problem \eqref{eq:control-problem} with the setting above for $d=30$ and $d=100$, and plot the results in Figure \ref{fig:obstacle-ex}.
Figure \ref{fig:obstacle-ex} shows the agent distributions at different time spots $t=0$, $0.25$, $0.5$, $0.75$, and $1$ in the $(x_1,x_2)$ coordinates for domain dimension $d=30$ (left plot) and $d=100$ (right plot).

In both plots in Figure \ref{fig:obstacle-ex}, we see the control vector field $u$ obtained by Algorithm \ref{alg:density-control} can successfully steer the agents around the obstacle to reach the target location $x^*$. 
Note that the result demonstrates that the proposed Algorithm \ref{alg:density-control} is scalable to cases where dimension $d$ can be as high as $100$. 

\begin{figure}[htb]
    \centering
    \includegraphics[scale=0.5]{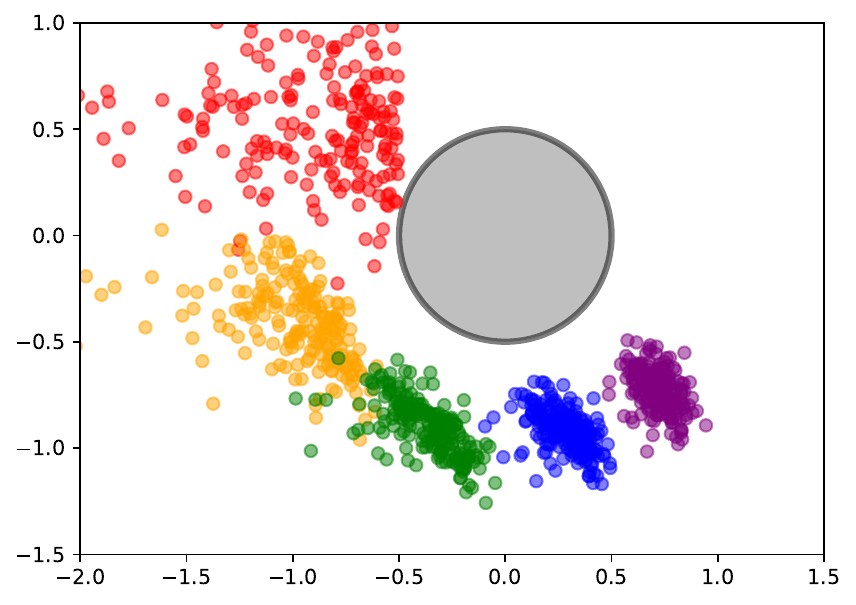}
    \includegraphics[scale=0.5]{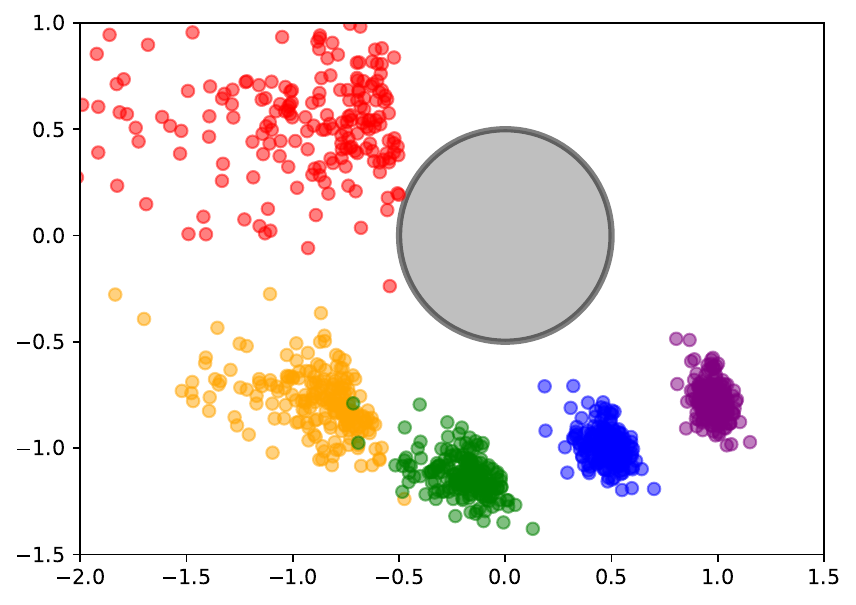}
    \caption{(Test 2) Visualization of agent movements under $u$ solved by Algorithm \ref{alg:density-control} for the probability control problem with cylindrical obstacle \eqref{eq:R-cylindrical-obstacle}. Agent distributions are show by colored circles at different times $t=0$ (red), $0.25$ (orange), $0.5$ (green), $0.75$ (blue), and $1$ (purple). The plots show the agent distributions in the $(x_1,x_2)$ coordinates for dimensions $d=30$ (left plot) and $d=100$ (right plot).}
    \label{fig:obstacle-ex}
\end{figure}

\subsection{Test 3 with Squeezing Obstacle and Agent Interaction}

In this experiment, we consider the control problem with both agent interaction and an obstacle. We set $\Omega = \Rbb^{d}$ where $d=30$. We place a squeezing obstacle which is formed by two wedges, and the agents need to move from their original locations to the target location $x^*=(2,0,\dots,0) \in \Rbb^{d}$ by passing through the gate between the two wedges..
The shape of the wedge and the agent locations in the $(x_{1},x_{2})$ coordinates are shown in Figure \ref{fig:squeeze-ex}.

In this test, we define
\begin{equation*}
    b(x)=
    \begin{cases}
        50(5x_1^2 - x_2^2 -0.1), & \text{if} \ 5x_{1}^{2}-x_{2}^{2} +0.1 \ge 0, \\
        0, & \text{elsewhere} .
    \end{cases}
\end{equation*}
%
% \[
% b(x) = 50\,\mathrm{ReLU}(5x_1^2 - x_2^2 + 0.1) \quad \mbox{and} \quad \nu=0.1\ .
% \]
We set the running reward functional $R$ as \eqref{eq:R-cylindrical-obstacle}, and also add an interaction term identical to the second term on the right-hand side of \eqref{eq:r-collision} with penalty weight $\gamma$. 
We use the same terminal reward functional $G$ as in Test 2 with $x^{*}$ chosen as above.
We choose an initial density $p(x)= \mathcal{N}(x\,;\,(-2,0,\dots,0),0.5I_d)$ and again set the terminal time to $T=1$. The results for the cases without this interaction ($\gamma=0$) and with this interaction ($\gamma=1$) are shown in the left and right plots, respectively, in Figure \ref{fig:squeeze-ex}.

We observe that the agents can squeeze through the gate between the two wedges in both plots of Figure \ref{fig:squeeze-ex}, and larger $\gamma$ causes the agents keep distances from each other when they move. In particular, they slowly start spreading out again as they exit the gate when $\gamma = 1$, as shown in the right plot. 

\begin{figure}[htb]
    \centering
    \includegraphics[scale=0.5]{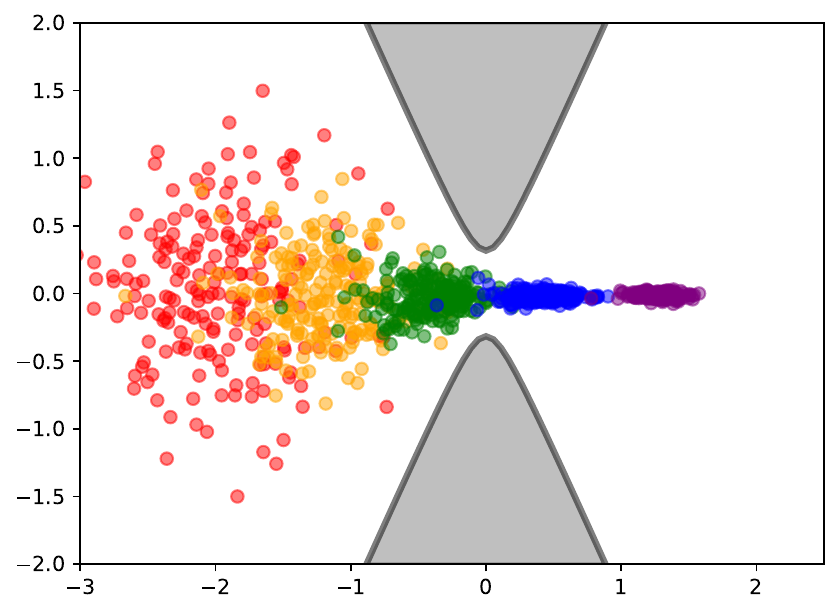}
    \includegraphics[scale=0.5]{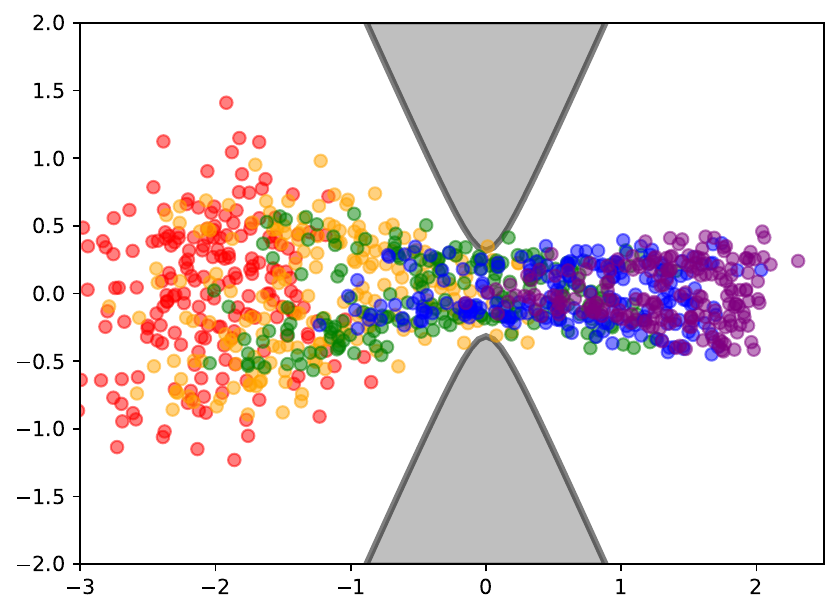}
    \caption{(Test 3) Visualization of agent movements under $u$ solved by Algorithm \ref{alg:density-control} for the probability control problem with both wedge obstacles and agent interaction. Agent distributions are show by colored circles at different times $t=0$ (red), $0.25$ (orange), $0.5$ (green), $0.75$ (blue), and $1$ (purple). The plots show the agent distributions in the $(x_1,x_2)$ coordinates with penalty weight $\gamma=0$ (left plot) and $\gamma=1$ (right plot) of the interaction term.}
    \label{fig:squeeze-ex}
\end{figure}

% \begin{figure}[htb]
%     \centering
%     \includegraphics[width=0.4\linewidth]{figures/combined_error.pdf}
%     \includegraphics[width=0.4\linewidth]{figures/combined_hamiltonian.pdf}
%     \caption{(Left) the cost functional value $I[u_{\theta_k}]$ versus iteration $k$ from 0 to $K=100$; (Right) the ratio $H_t^K/|H_0^K|$ versus time $t$. In both plots, Test 1 with interaction only is shown by the red curves, Test 2 with the cylinder obstacle by blue curves, and Test 3 with both interaction and the squeezing obstacle by green curves.}
%     \label{fig:ham-and-err-comboined}
% \end{figure}

\section{Conclusion}
\label{sec:conclusion}

In this paper, we present a comprehensive theoretical framework of optimal probability density control in the standard measure spaces. Specifically, we establish two fundamental results in this setting: the maximum principle of optimal control vector fields and the Hamilton--Jacobi--Bellman equation of value functional. The results are concise, easy to interpret, and supported by rigorous mathematical analysis. This new framework does not requires Wasserstein metrics typically used in the literature. This allows substantial simplifications of numerical computations in practice. We present a scalable algorithm aligned with the maximum principle and demonstrate its promising performance on several high-dimensional multi-agent control problems.

\bibliographystyle{abbrv}
\bibliography{library}

\end{document}